\documentclass[12pt]{amsart}
\usepackage{epsfig,comment}
\usepackage{mathtools}
\usepackage[subnum]{cases}
\usepackage{enumerate}
\usepackage{bbm,bm}
\usepackage{amsmath,amssymb,mathrsfs}
\usepackage{color}
\usepackage[bookmarks=true,
bookmarksnumbered=true, breaklinks=true,
pdfstartview=FitH, hyperfigures=false,
plainpages=false, naturalnames=true,
colorlinks=true,pagebackref=true,
pdfpagelabels,colorlinks=true,
	linkcolor=blue,
	citecolor=black,
	filecolor=black,
	urlcolor=black]{hyperref}
\usepackage{wasysym}

\usepackage{pgfplots}
\pgfplotsset{width=10cm,compat=1.9}
\usepgfplotslibrary{fillbetween}

\usepgfplotslibrary{external}
\newtheorem{theorem}{Theorem}[section]
\newtheorem{definition}[theorem]{Definition}
\newtheorem{lemma}[theorem]{Lemma}
\newtheorem{corollary}[theorem]{Corollary}
\newtheorem{remark}[theorem]{Remark}

\newcommand{\R}{\mathbb{R}} 
\newcommand{\E}{\mathbb{E}}
 
\newcommand{\Sp}{\mathbb{S}^{d-1}}
\newcommand{\dint}{\mathrm{d}}
\newcommand{\NN}{N^{-\frac{2}{d-1}}}

\DeclareMathOperator{\vol}{vol}

\DeclareMathOperator{\Gr}{Gr}

\DeclareMathOperator{\dist}{dist}

\newcommand{\bbinom}[2]{\left[\!\!\begin{array}{c}#1\\#2\end{array}\!\!\right]}

\title[Approximation of the Euclidean ball by polytopes]{Approximation of the Euclidean ball by polytopes\\with a fixed number of $k$-faces}

\author{Steven Hoehner, Carsten Sch\"utt and Elisabeth M. Werner}
\date{\today}

\begin{document}

\setcounter{footnote}{0}

\begin{abstract}\noindent
We  derive lower estimates for the approximation of the $d$-dimensional Euclidean ball by polytopes with a fixed number of $k$-dimensional faces, $k\in\{0,1,\ldots,d-1\}$. The metrics considered include the intrinsic volume difference and the Hausdorff metric. In the case of inscribed and circumscribed polytopes, our main results extend the previously obtained bounds from $k=0$ and $k=d-1$, respectively, to half of the $f$-vector of the approximating polytope. For arbitrarily positioned polytopes, we also improve a special case of a result of K. J. B\"or\"oczky ({\it J. Approx. Theory}, 2000) by a factor of dimension. This paper addresses a question of P. M. Gruber ({\it Convex and Discrete Geometry}, p. 216), who asked for results on the approximation of convex bodies by polytopes with a fixed number of $k$-faces when $1\leq k\leq d-2$.
\end{abstract}

\maketitle

\renewcommand{\thefootnote}{}

\renewcommand{\thefootnote}{\arabic{footnote}}
\setcounter{footnote}{0}

%%%%%%%%%%%%%%%%%%%%%%%%%%%%%%%%%
\section{Introduction}

The approximation of convex bodies by polytopes is a fundamental problem in convex and discrete geometry. This problem is concerned with quantifying how well a complex object, such as the $d$-dimensional Euclidean ball, can be approximated by a simpler geometric shape, such as a polytope. Historically, two main models have been used to study this problem:
\begin{itemize}
    \item {\bf Inscribed polytopes with a fixed number of vertices.} In this model, the polytope is the convex hull of a finite set of points that are all contained within the convex body.

    \item {\bf Circumscribed polytopes with a fixed number of facets.} In this model, the polytope is the intersection of a finite number of halfspaces, each of which contains the convex body. A common approach is to select the normals of these halfspaces from the boundary of the convex body.
\end{itemize}

While extensive research exists for these two standard models, there is a significant gap in the literature regarding the approximation of convex bodies by polytopes with a fixed number of $k$-faces, where $1\leq k\leq d-2$. This is a long-standing question posed by P. M. Gruber in his book \cite{GruberBook} (see page 216). The present paper addresses Gruber's question directly by providing new lower bounds on the approximation of the Euclidean ball by polytopes with a fixed number of $k$-faces. Our main results extend previously established bounds for inscribed and circumscribed polytopes (which correspond to $k=0$ and $k=d-1$, respectively, in the standard models of approximation) to a broad range of intermediate $k$-faces. 
For the intermediate range $1\leq k\leq d-2$, a few results are known. In dimension $d=3$, asymptotic formulas have been determined for the approximation of smooth convex bodies by inscribed polytopes with a restricted number of edges under the Hausdorff metric \cite{Boroczky-Fodor-Vigh} and the symmetric difference metric \cite{Boroczky-Gomis-Tick}. In general dimensions $d\geq 4$, however, very little is known. K. J. B\"or\"oczky \cite{Boroczky-2000} derived lower bounds for the symmetric difference metric and Hausdorff distance of a smooth convex body and an arbitrarily positioned polytope with a restricted number of $k$-faces. Besau, Hoehner and Kur \cite{BHK} provided lower bounds for the asymptotic best approximation of the Euclidean ball by inscribed or circumscribed polytopes with a fixed number of $k$-faces under the intrinsic volume difference, under the assumption that the polytopes are simplicial. Gruber also posed related problems on typical faces of best-approximating polytopes in asymptotic approximations \cite{GruberIV, Gruber2001}. There are some results in this direction when the number of edges is restricted, see, e.g., \cite{BoroczkyTickWintsche2007,Vigh2009} as well as \cite{LRSW}.

In this paper, we provide answers to these questions and give optimal estimates for the approximation of the Euclidean ball by polytopes with a restricted number of $k$-faces. Our methods combine tools from the combinatorial theory of polytopes, polar duality, and  isoperimetric inequalities. Our results extend and generalize the aforementioned results of \cite{BHK}, removing the assumption that the approximating polytopes are simplicial while extending the range of $k$ considered. We also provide lower bounds for the symmetric difference metric and surface area deviation of the Euclidean ball and an arbitrarily positioned polytope for $k\in\{\lceil d/2\rceil-1,\ldots,d-1\}$. This extends results of Ludwig, Sch\"utt, and Werner \cite{LSW} and of Hoehner, Sch\"utt, and Werner \cite{HSW}, respectively, from $k=d-1$ to all $k\in\{\lceil d/2\rceil-1,\ldots,d-1\}$. In particular, in the special case of the Euclidean ball and for $k\in\{\lceil d/2\rceil-1,\ldots,d-1\}$, our result on the symmetric difference metric improves an estimate of B\"or\"oczky \cite{Boroczky-2000} by a factor of dimension. We also provide lower bounds for the Hausdorff approximation of the Euclidean ball by inscribed or circumscribed polytopes with a restricted number of $k$-faces, for any $k\in\{0,1,\ldots,d-1\}$.

%%%%%%%%%%%%%%%%%%%%
\subsection{Overview of the paper}

First, in Subsection \ref{background}, we briefly discuss the background and notation that will be used throughout the paper. Our main results are stated in Section \ref{mainresults-sec}, and the proofs of our main results are in Sections \ref{mw-vol-sec}, \ref{asymptotic-section} and  \ref{mainThm-sec}. Next, in Section \ref{hausdorff-sec}, we state and prove results on the Hausdorff approximation of the Euclidean ball by inscribed or circumscribed polytopes with a fixed number of $k$-faces. In Section \ref{sec-discussion}, we briefly discuss some open problems and conjectures. We conclude the paper with an appendix in Section \ref{sec-appendix}, which contains the proofs of some lemmas that are used in the proofs of some of our main results.

\subsection{Background and notation}\label{background}

The standard inner product on $\R^d$ is written $\langle x,y\rangle=\sum_{i=1}^d x_i y_i$, where $x=(x_1,\ldots,x_d), y=(y_1,\ldots,y_d)\in\R^d$. The Euclidean norm of $x\in\R^d$ is $\|x\|_2=\sqrt{\langle x,x\rangle}$. The $d$-dimensional Euclidean unit ball in $\R^d$ centered at the origin $o$ is $B_d=\{x\in\R^d:\,\|x\|_2\leq 1\}$. Its boundary, $\partial B_d$, is the unit sphere in $\R^d$, i.e., $\Sp=\{x\in\R^d:\,\|x\|_2=1\}$. We write $\kappa_j$, $1 \leq j \leq d$, for the $j$-dimensional volume of the $j$-dimensional Euclidean unit ball $B_j$ and $\omega_j$ for the  $(j-1)$-dimensional volume of $\partial B_j$.
Throughout the paper, we let $c,c_1,c_2,C,\ldots$ denote positive absolute constants whose values may change from line to line. If a positive constant depends on the dimension $d$, the number of $k$-faces $M$ of a polytope, or some other parameter, we will denote this dependence explicitly. For example, $c(d)$ denotes a positive constant that only depends on the dimension.
   
A polytope is the convex hull of a finite point set in $\R^d$. Let $P$ be a polytope in $\R^d$. A set $F\subset P$ is a \emph{(proper) face} of $P$ if $F=P\cap H$, where $H$ is a supporting hyperplane of $P$. If the dimension of a face $F$ of $P$ is $k\in\{0,1,\ldots,d-1\}$, then we say that $F$ is a \emph{$k$-dimensional face} (or \emph{$k$-face})  of $P$. In particular, the 0-dimensional, 1-dimensional and $(d-1)$-dimensional faces of $P$ are called the vertices, edges and facets of $P$, respectively.  For $k\in\{0,1,\ldots,d-1\}$ and a polytope $P$ in $\R^d$,  let $\mathcal{F}_k(P)$ denote the set of all $k$-dimensional faces of $P$. For more background on convex polytopes, we refer the reader to, for example, the books \cite{Brondsted,Grunbaum}.

A key ingredient in our proofs is the following pair of combinatorial inequalities due to Hinman \cite{Hinman1,Hinman2}.  Let $P$ be a convex polytope in $\R^d$ and for $k\in\{0,1,\ldots,d-1\}$, let $f_k(P)$ denote the number of $k$-dimensional faces of $P$. \begin{align}
f_k(P) &\geq  f_0(P),\quad k\in\left\{0,1,\ldots,\lfloor\tfrac{d}{2}\rfloor\right\} \label{vertices-bd}\\
f_k(P) &\geq f_{d-1}(P), \quad k\in\left\{\lceil\tfrac{d}{2}\rceil-1,\ldots, d-1\right\}. \label{facets-bd}
\end{align}

A convex body is a convex, compact subset of $\R^d$ with nonempty interior. We now recall several notions related to convex bodies. For details we refer to, for example, the book \cite{Schneiderbook}. The intrinsic volumes of a convex body $K$ in $\R^d$ are defined as the coefficients in Steiner's formula:
\[
\forall\varepsilon\geq 0,\quad \vol_d(K+\varepsilon B_d)=\sum_{j=0}^d \varepsilon^{d-j} \kappa_{d-j} V_j(K).
\]
In particular, $V_d(K)=\vol_d(K)$ is the volume of $K$, $V_{d-1}(K)=\frac{1}{2}\vol_{d-1}(\partial K)$ is half the surface area of $K$, and $V_1(K)$ is a constant multiple of the mean width of $K$. Recall that the mean width $w(K)$ of $K$ is 
\begin{equation}\label{width}
w(K)=\frac{2}{\omega_d}\int_{\Sp}h_K(u)\,du,
\end{equation} 
where $h_K(u)=\max_{x\in K}\langle x,u\rangle$ is the support function of $K$ in the direction $u\in\Sp$. Then  $V_1(K)=\frac{\omega_{d}}{2 \kappa_{d-1}}w(K)$.

Kubota's integral formula states that
\begin{equation}\label{kubota}
    V_j(K)=\bbinom{d}{j}
        \int_{\Gr(d,j)}\vol_j(K|H)\,\dint \nu_j(H),
\end{equation}
where $\Gr(d,j)$ is the Grassmannian of all $j$-dimensional subspaces of $\R^d$, $K|H$ is the orthogonal projection of $K$ onto $H$, $\nu_j$ is the unique Haar probability measure on $\Gr(d,j)$, and 
\[
\bbinom{d}{j} := \binom{d}{j} \frac{\kappa_d}{\kappa_{j} \kappa_{d-j}}
\]
denotes the \emph{flag coefficient} of Klain and Rota \cite{Klain-Rota-book}. 

The intrinsic volumes of a convex body $K$ in $\R^d$ satisfy the extended isoperimetric inequality, which states that
\begin{equation}\label{extisoineq}
\forall j\in\{1,\ldots,d\},\quad \left(\frac{\vol_d(K)}{\kappa_d}\right)^{1/d}\leq \cdots \leq \left(\frac{V_j(K)}{V_j(B_d)}\right)^{1/j} \leq\cdots\leq \frac{V_1(K)}{V_1(B_d)}.
\end{equation}
For more background on intrinsic volumes, we refer the reader, again,  to the book \cite{Schneiderbook}.

For convex bodies $K$ and $L$ in $\R^d$ and $j\in\{1,\ldots,d\}$, Besau and Hoehner \cite{BH-2022} defined the $j$th \emph{intrinsic volume metric} $\delta_j(K,L)$ by
\begin{equation}\label{def:metric}
    \delta_j(K,L) := \bbinom{d}{j} \int_{\Gr(d,j)} \vol_j((K|H)\triangle (L|H))\, \dint\nu_j(H).
\end{equation}
Here $A\triangle B$ denotes the symmetric difference of sets $A$ and $B$. In particular, $\delta_d(K,L)=\vol_d(K\triangle L)$ is the symmetric difference metric. It was shown in \cite[Theorem 2.1]{BH-2022} that $\delta_j$ is a metric on the class $\mathcal{K}^d$ of convex bodies in $\R^d$, is continuous with respect to the Hausdorff metric, is rigid motion invariant and is positively $j$-homogeneous.  In the special case $K\subset L$, \eqref{def:metric} reduces to the \emph{intrinsic volume difference}
\[
\delta_j(K,L) = V_j(L)-V_j(K).
\]
Note that if $K\subset L$, then $\delta_1(K,L)$ is a constant multiple of the mean width difference
\[
\Delta_w(K,L) := w(L)-w(K).
\]

The intrinsic volume metric has been studied in the approximation of convex bodies by polytopes. For more background, we refer the reader to  \cite{BH-2022, BHK} and the references therein. For more background on the asymptotic best approximation of convex bodies by polytopes, see, for example, \cite{BH-2022,Bronshtein-survey,Hoehner-survey,HK-DCG,HSW,Hoehner2016,Kur2017,Ludwig1999,LSW} and the references therein. 

%%%%%%%%%%%%%%%%%%%%%%%%%%%%%
\section{Main Results}\label{mainresults-sec}

\subsection{General lower bounds}

The next two results are interesting because they are nonasymptotic in nature; there is no minimum number of $k$-faces required to obtain these bounds.

\begin{theorem}\label{inscribed-mw-cor-kfaces}
    Let $d\geq 2$ and $k\in\{0,1,\ldots,\lfloor\frac{d}{2}\rfloor\}$. For all polytopes $P_M\subset B_d$ with at most $M$ $k$-faces that contain $o$ in their interiors, we have
    \[
    \Delta_w(B_d,P_M) \geq \frac{1}{4}\left(\frac{\omega_d}{4  \kappa_{d-1}}\right)^{\frac{2}{d-1}}M^{-\frac{2}{d-1}}.
    \]
\end{theorem}

\begin{theorem}\label{circumscribed-vol-cor-kfaces}
    Let $d\geq 2$ and $k\in\{\lceil d/2\rceil-1,\ldots,d-1\}$. For all polytopes $P_M\supset B_d$ with at most $M$ $k$-faces, we have
    \[
    \vol_d(P_M\setminus B_d) \geq \frac{\omega_{d}}{8}\left(\frac{\omega_{d}}{4 \kappa_{d-1}}\right)^{\frac{2}{d-1}}M^{-\frac{2}{d-1}}.
    \]
\end{theorem}
Note that by Stirling's formula,
\[
\left(\frac{\omega_{d}}{4 \kappa_{d-1}}\right)^{\frac{2}{d-1}}=1+O\left(\frac{\ln d}{d}\right).
\]
Theorems \ref{inscribed-mw-cor-kfaces} and \ref{circumscribed-vol-cor-kfaces} are proved in Section \ref{mw-vol-sec}. The absolute constants $1/4, 1/8$, etc. appearing in these theorems are not the best possible and can be optimized; see Section \ref{asymptotic-section} for further explanation.

Using the extended isoperimetric inequality (\ref{extisoineq}), we derive the following estimates for the intrinsic volume difference of the Euclidean ball and an inscribed or circumscribed polytope. 

\begin{theorem}\label{mainThm}
        Let  $d\geq 2$ be given, and fix $j\in \{1,\ldots,d\}$. 
        \begin{itemize}
        \item[(i)]
        Let  $k\in\{0,1,\ldots,\lfloor d/2\rfloor\}$. For all polytopes $P_M\subset B_d$ with at most $M$ $k$-faces and with $o\in\operatorname{int}(P_M)$, we have
    \begin{align*}
    \delta_j(P_M,B_d) \geq \frac{j}{4}\left(\frac{\omega_{d}}{4\kappa_{d-1}}\right)^{\frac{2}{d-1}}V_j(B_d)M^{-\frac{2}{d-1}}
     \left(1-\frac{j-1}{8}\left(\frac{\omega_{d}}{4\kappa_{d-1}}\right)^{\frac{2}{d-1}}M^{-\frac{2}{d-1}}\right). 
    \end{align*}
In particular, let $c\in(0,1)$. Then for all $M\geq \frac{\omega_{d}}{4\kappa_{d-1}}\left(\frac{j-1}{8(1-c)}\right)^{\frac{d-1}{2}}$, we have
    \[
   \delta_j(P_M,B_d)\geq \frac{c\cdot j}{4}\left(\frac{\omega_{d}}{4\kappa_{d-1}}\right)^{\frac{2}{d-1}}V_j(B_d)M^{-\frac{2}{d-1}}.
    \]
        
        \item[(ii)] Let $k\in\{\lceil d/2\rceil-1,\ldots,d-1\}$. For all polytopes $P_M\supset B_d$ with at most $M$ $k$-faces, we have
     \begin{align*}
    \delta_j(P_M,B_d) \geq \frac{j}{8}\left(\frac{\omega_{d}}{4\kappa_{d-1}}\right)^{\frac{2}{d-1}}V_j(B_d)M^{-\frac{2}{d-1}}
     \left(1-\frac{d-j}{16}\left(\frac{\omega_{d}}{4\kappa_{d-1}}\right)^{\frac{2}{d-1}}M^{-\frac{2}{d-1}}\right). 
    \end{align*}
In particular, let $c\in(0,1)$. Then for all $M\geq \frac{\omega_{d}}{4\kappa_{d-1}}\left(\frac{d-j}{16(1-c)}\right)^{\frac{d-1}{2}}$, we have
    \[
   \delta_j(P_M,B_d)\geq \frac{c\cdot j}{8}\left(\frac{\omega_{d}}{4\kappa_{d-1}}\right)^{\frac{2}{d-1}}V_j(B_d)M^{-\frac{2}{d-1}}.
    \]
\end{itemize}
\end{theorem}

\noindent Theorem \ref{mainThm} is proved in Section \ref{mainThm-sec}.

\begin{remark}
    For the known cases $k\in\{0,d-1\}$, comparing the inequalities in Theorem \ref{mainThm} with the asymptotic estimates in \cite[Theorem 1]{BHK}, we see that our inequalities give the best possible order in these cases, up to  absolute constant factors. 
\end{remark}

\begin{remark}
Note that the lower bounds in Theorem \ref{mainThm} hold for all $M\geq d+1$ in the following endpoint ranges of $j$. In Theorem \ref{mainThm}(i),  this occurs for small values of $j$, for instance when $j\leq 1+8(1-c)$. In Theorem \ref{mainThm}(ii), it occurs for large values of $j$, for instance when $j\geq d-16(1-c)$. Indeed, in these ranges the corresponding threshold for $M$ is bounded above by $\frac{\omega_d}{4\kappa_{d-1}}$, 
which is less than $d+1$ for every $d\geq 2$.
    %Note that for large values of $j$, i.e., $j> d-8(1-c)$ in Theorem \ref{mainThm}(i) and $j> d-16(1-c)$ in Theorem \ref{mainThm}(ii), since $\omega_{d}/\kappa_{d-1} \geq \sqrt{2\pi d}(1-1/d)$, for any such $j$ the results hold for all $M\geq d+1$.
\end{remark}

\begin{remark}
Gordon, Reisner, and Sch\"utt \cite{GRS-umbrellas} showed that there are absolute constants $c_1,c_2>0$ such that for every $d\geq 2$ and every $N\geq (c_1 d)^{\frac{d-1}{2}}$, for all polytopes $P_N\subset B_d$ with at most $N$ vertices,
\begin{equation}\label{GRS-eqn}
\vol_d(B_d\setminus P_N) \geq c_2 d\kappa_d\NN.
\end{equation}
The case $k=0$ and $j=d$ in Theorem \ref{mainThm}(i) gives a lower bound of the same order under the same large-$N$ hypothesis, for inscribed polytopes that contain $o$ in their interiors. Indeed, since $\delta_d(P_N,B_d)=\vol_d(B_d\setminus P_N)$ for $P_N\subset B_d$, Theorem \ref{mainThm}(i) yields
\[
\vol_d(B_d\setminus P_N)
\geq
\frac{d\kappa_d}{4}\left(\frac{\omega_d}{4\kappa_{d-1}}\right)^{\frac{2}{d-1}}\NN
\left(1-\frac{d-1}{8}\left(\frac{\omega_d}{4\kappa_{d-1}}\right)^{\frac{2}{d-1}}\NN
\right).
\]
In particular, taking $c=1/2$ in Theorem \ref{mainThm}(i), for all $N\geq \frac{\omega_d}{4\kappa_{d-1}}\left(\frac{d-1}{4}\right)^{\frac{d-1}{2}}$, this gives
\[
\vol_d(B_d\setminus P_N)\geq cd\kappa_d\NN
\]
with an absolute constant $c>0$, which is the same order as \eqref{GRS-eqn}. Moreover, by Stirling's inequality the threshold in (i) can be expressed as $N\geq (Cd)^{\frac{d-1}{2}}$ for some absolute constant $C>0$, which is the same large-$N$ hypothesis as in \cite{GRS-umbrellas}. 
\end{remark}

\begin{comment}
\begin{remark}
    Gordon, Reisner and Sch\"utt \cite{GRS-umbrellas} showed that there are absolute constants $c_1,c_2>0$ such that for every $d\geq 2$ and every $N\geq (c_1 d)^{\frac{d-1}{2}}$, for all polytopes $P_N\subset B_d$ with at most $N$ vertices,
    \begin{equation}\label{GRS-eqn}
\vol_d(B_d\setminus P_N) \geq c_2 d\kappa_d\NN.
    \end{equation}
    Prochno, Sch\"utt and Werner \cite{PSW-2022} conjectured that the lower bound \eqref{GRS-eqn} still holds even without the assumption $N\geq (c_1 d)^{\frac{d-1}{2}}$.  Thus, the special case $j=d$, $k=0$ in Theorem \ref{mainThm}(i) gives a corresponding lower bound for the Euclidean ball. %Thus, Theorem \ref{circumscribed-vol-cor-kfaces} (or the special case $j=d$, $k=0$ in Theorem \ref{mainThm}(i))  confirms this conjecture. 
\end{remark}
\end{comment}
%%%%%%%%%%%%%%%%%%%%%%%%%%%

\subsection{Inequalities for asymptotic best approximation}

Let $\mathcal{P}^{\rm in}_{d,k}(N)$ (respectively, $\mathcal{P}^{\rm out}_{d,k}(N)$) denote the set of all polytopes $P\subset B_d$ (respectively, $P\supset B_d$) such that $P$ has at most $N$ $k$-faces. We also define  the subsets $\mathcal{P}_{d,k}^{\rm in,s}(N):=\{P\in\mathcal{P}^{\rm in}_{d,k}(N):\,P\text{ is simplicial}\}$ and  $\mathcal{P}_{d,k}^{\rm out,s}(N):=\{P\in\mathcal{P}^{\rm out}_{d,k}(N):\,P\text{ is simplicial}\}$.

\begin{lemma}\label{inclusion-chain}
    For an integer $d\geq 2$ and $k\in\{0,1,\ldots,d-1\}$, let \[
    \mathscr{C}_{d,k}(N)\in\{\mathcal{P}^{\rm in}_{d,k}(N),\mathcal{P}^{\rm out}_{d,k}(N),\mathcal{P}^{\rm in,s}_{d,k}(N),\mathcal{P}^{\rm out,s}_{d,k}(N)\}.
    \]
    \begin{itemize}
        \item[(i)]  For any $k\in\{0,1,\ldots,\lfloor d/2\rfloor\}$, we have
    \[
\mathscr{C}_{d,0}(N) \supset \mathscr{C}_{d,k}(N).
    \]

    \item[(ii)] For any $k\in\{\lceil d/2\rceil-1,\ldots,d-1\}$, we have
    \[
\mathscr{C}_{d,d-1}(N) \supset \mathscr{C}_{d,k}(N).
    \]
    \end{itemize}
    
\end{lemma}

\begin{proof}
    (i) Let $k\in\{0,1,\ldots,\lfloor d/2\rfloor\}$ be fixed. Take any $P\in\mathscr{C}_{d,k}(N)$. Then by definition, $f_k(P)\leq N$. Thus, by \eqref{vertices-bd}, we have $f_0(P)\leq f_k(P)\leq N$. Hence, $P$ has at most $N$ vertices, so $P\in\mathscr{C}_{d,0}(N)$. This proves (i). Part (ii) is shown in the same way, only now we use \eqref{facets-bd} instead of \eqref{vertices-bd}.
\end{proof}

In what follows, we set
\[
\delta_j(B_d,\mathscr{C}_{d,k}(N)):=\inf_{P_N\in\mathscr{C}_{d,k}(N)}\delta_j(B_d,P_N).
\]
It follows from a compactness argument that the infimum is achieved for some \emph{best-approximating polytope} $P_N^*\in\mathscr{C}_{d,k}(N)$.

 It was shown by Besau, Hoehner and Kur \cite[Corollary 2 ii)]{BHK} that for any fixed $k\in\{0,1,\ldots,\lfloor d/2\rfloor\}$, when considering simplicial polytopes the following asymptotic lower bound holds true:
\begin{equation}\label{BHK-asymptotic}
\liminf_{N\to\infty}\left\{N^{\frac{2}{d-1}}\delta_j(B_d,\mathcal{P}_{d,k}^{\rm in,s}(N)) \right\} \geq \frac{1}{2}\operatorname{div}_{d-1}\omega_{d}^{\frac{2}{d-1}}jV_j(B_d).
    \end{equation}
    In \cite[Corollary 2 iii)]{BHK}, they proved that for any fixed $k\in\{\lfloor 3(d-1)/4\rfloor,\ldots,d-1\}$, for circumscribed simplicial polytopes we have
    \begin{equation}\label{BHK-asymptotic-2}
\liminf_{N\to\infty}\left\{N^{\frac{2}{d-1}}\delta_j(B_d,\mathcal{P}_{d,k}^{\rm out,s}(N)) \right\} \geq \frac{1}{2}\operatorname{del}_{d-1}\omega_{d}^{\frac{2}{d-1}}jV_j(B_d).
    \end{equation}
    Here $\operatorname{del}_{d-1}$ and $\operatorname{div}_{d-1}$ are positive constants that depend only on the dimension $d$. They are connected with Delone triangulations and Dirichlet--Voronoi tilings in $\R^d$, respectively, and arise in formulas for the asymptotic best approximation of smooth convex bodies by polytopes. The best known asymptotic estimate for $\operatorname{del}_{d-1}$ is due to Mankiewicz and Sch\"utt \cite{MaS1}, who proved that  $\operatorname{del}_{d-1}=(2\pi e)^{-1}d+O(\ln d)$. An almost-sharp estimate for $\operatorname{div}_{d-1}$ was given by Hoehner and Kur \cite[Theorem 1.4]{HK-DCG}, who proved that $|\operatorname{div}_{d-1}-(2\pi e)^{-1}(d+\ln d)|=O(1)$. We give a new simplified proof of this result in Section \ref{asymptotic-section}.
    
In the next result, we drop the restriction that the approximating polytopes are simplicial and show that \eqref{BHK-asymptotic} and \eqref{BHK-asymptotic-2} still hold true; moreover, we extend the range of $k$ in \eqref{BHK-asymptotic-2} to $k\in\{\lceil d/2\rceil-1,\ldots,d-1\}$, which is about 25\% more of the total $f$-vector of the approximating circumscribed polytope.

\begin{corollary}\label{large-N-thm}
Let $d\geq 2$ and fix $j\in\{1,\ldots,d\}$.
   \begin{itemize} 
   \item[(i)] Let $k\in\{0,1,\ldots,\lfloor d/2\rfloor\}$ be fixed. Then
    \[
\liminf_{N\to\infty}\left\{N^{\frac{2}{d-1}}\delta_j(B_d,\mathcal{P}_{d,k}^{\rm in}(N)) \right\} \geq \frac{1}{2}\operatorname{div}_{d-1}\omega_{d}^{\frac{2}{d-1}}jV_j(B_d).
    \]

    \item[(ii)]  Let $k\in\{\lceil d/2\rceil-1,\ldots,d-1\}$ be fixed. Then
    \[
\liminf_{N\to\infty}\left\{N^{\frac{2}{d-1}}\delta_j(B_d,\mathcal{P}_{d,k}^{\rm out}(N)) \right\} \geq \frac{1}{2}\operatorname{del}_{d-1}\omega_{d}^{\frac{2}{d-1}}jV_j(B_d).
    \]
    \end{itemize}
\end{corollary}

\begin{proof}
(i) It was shown in the proof of \cite[Theorem 1 i)]{BHK} that in the case of inscribed polytopes with a restricted number of vertices, the following asymptotic inequality holds true: 
     \[
\liminf_{N\to\infty}\left\{N^{\frac{2}{d-1}}\delta_j(B_d,\mathcal{P}^{\rm in}_{d,0}(N)) \right\} \geq \frac{1}{2}\operatorname{div}_{d-1}\omega_{d}^{\frac{2}{d-1}}jV_j(B_d).
    \]
    By Lemma \ref{inclusion-chain}, for any $k\in\{0,1,\ldots,\lfloor d/2\rfloor\}$, we have $\mathcal{P}_{d,0}^{\rm in}(N) \supset \mathcal{P}_{d,k}^{\rm in}(N)$, which implies
    \[
 \delta_j(B_d,\mathcal{P}_{d,0}^{\rm in}(N)) \leq \delta_j(B_d,\mathcal{P}_{d,k}^{\rm in}(N)).
    \]
Therefore,
\begin{align*}
    \liminf_{N\to\infty}\left\{N^{\frac{2}{d-1}}\delta_j(B_d,\mathcal{P}_{d,k}^{\rm in}(N)) \right\}\geq \liminf_{N\to\infty}\left\{N^{\frac{2}{d-1}}\delta_j(B_d,\mathcal{P}_{d,0}^{\rm in}(N)) \right\} 
    \geq \frac{1}{2}\operatorname{div}_{d-1}\omega_{d}^{\frac{2}{d-1}}jV_j(B_d).
\end{align*}
Part (ii) is shown in the same way. 
\end{proof}

%%%%%%%%%%%%%%%%%%%%%%%%%%%%%
\subsection{The total intrinsic volume metric}

The \emph{Wills functional} (or \emph{total intrinsic volume}) of a convex body $K$ in $\R^d$  is defined by $W(K)=\sum_{j=0}^d V_j(K)$. For convex bodies $K$ and $L$ in $\R^d$, Besau, Hoehner and Kur \cite{BHK} defined the \emph{Wills deviation} (or \emph{total intrinsic volume deviation})  $\Delta_\Sigma(K,L)$ to be
\[
\Delta_\Sigma(K,L) := 
W(K)+W(L)-2W(K\cap L).
\]

The Wills deviation has found applications to the polytopal  approximation of convex bodies; see \cite{BHK} for more background, as well as for a more general stochastic form of the Wills deviation.

This leads us to make the following
\begin{definition}
    For convex bodies $K$ and $L$ in $\R^d$, the  \emph{total intrinsic volume metric}  (or \emph{Wills metric}) $\delta_\Sigma(K,L)$ is defined by
\[
\delta_\Sigma(K,L) := \sum_{j=0}^d \delta_j(K,L).
\]
\end{definition}

Note that $\delta_\Sigma$ is a metric on the set of convex bodies in $\R^d$ because it is a finite sum of metrics, and it is continuous with respect to the Hausdorff metric. Moreover, observe that if $K\subset L$, then
\[
\delta_\Sigma(K,L) = \Delta_\Sigma(K,L) = W(L)-W(K)
\]
reduces to the \emph{total intrinsic volume difference} of $K$ and $L$. Furthermore, since $(K\cap L)|H\subset (K|H)\cap(L|H)$ for any $H\in\Gr(d,j)$, we have $\delta_\Sigma\leq \Delta_\Sigma$.

In what follows, we let $\overline{W}(B_d):=\sum_{j=1}^d \frac{j}{d}V_j(B_d)$ denote the \emph{average intrinsic volume} of $B_d$. Average intrinsic volumes have  recently found applications to the polytopal approximation of convex bodies \cite{BHK} and to the  concentration of ultra log-concave distributions in probability \cite{Aravinda-et-al, Lotz-et-al}.

For the total intrinsic volume metric, we obtain the following result. 

\begin{corollary}\label{wills-cor}
 Let $d\geq 2$.
 \begin{itemize}
     \item[(i)] Fix $k\in\{0,1,\ldots,\lfloor d/2\rfloor\}$ and $c\in(0,1)$. For all $M \geq \frac{\omega_{d}}{4\kappa_{d-1}}\left(\frac{d-1}{8(1-c)}\right)^{\frac{d-1}{2}}$ and all polytopes $P_M\in\mathcal{P}_{d,k}^{\rm in}(M)$ with $o\in\operatorname{int}(P_M)$, we have
     \[
\delta_\Sigma(B_d,P_M) \geq \frac{c}{4}\left(\frac{\omega_{d}}{4\kappa_{d-1}}\right)^{\frac{2}{d-1}}d\overline{W}(B_d)M^{-\frac{2}{d-1}}.
     \]

     \item[(ii)] Fix $k\in\{\lceil d/2\rceil-1,\ldots,d-1\}$ and $c\in(0,1)$. For all $M \geq \frac{\omega_{d}}{4\kappa_{d-1}}\left(\frac{d-1}{16(1-c)}\right)^{\frac{d-1}{2}}$ and all polytopes $P_M\in\mathcal{P}_{d,k}^{\rm out}(M)$, we have
     \[
\delta_\Sigma(B_d,P_M) \geq \frac{c}{8}\left(\frac{\omega_{d}}{4\kappa_{d-1}}\right)^{\frac{2}{d-1}}d\overline{W}(B_d)M^{-\frac{2}{d-1}}.
     \]
 \end{itemize}
\end{corollary}

\begin{proof}
We only prove (i) as the proof of (ii) follows along the same lines. Let $P_M\in\mathcal{P}_{d,k}^{\rm in}(M)$ with $o\in\operatorname{int}(P_M)$. For fixed $k$, note that all $d$  inequalities in  Theorem \ref{mainThm}(i) (for $j\in\{1,\ldots,d\}$) hold when $M \geq \frac{\vol_{d-1}(\partial B_d)}{4\vol_{d-1}(B_{d-1})}\left(\frac{d-1}{8(1-c)}\right)^{\frac{d-1}{2}}$. Thus, for any such $M$, 
 \[
\delta_\Sigma(B_d,P_M) =\sum_{j=0}^d \delta_j(B_d,P_M) \geq \frac{c}{4}\left(\frac{\omega_{d}}{4\kappa_{d-1}}\right)^{\frac{2}{d-1}}d\overline{W}(B_d)M^{-\frac{2}{d-1}}.
     \]
\end{proof}

We also obtain the following asymptotic lower bounds.

\begin{corollary}\label{large-N-thm-Wills}
   Let $d\geq 2$.

   \begin{itemize} 
   \item[(i)] Let $k\in\{0,1,\ldots,\lfloor d/2\rfloor\}$ be fixed. Then
    \[
\liminf_{N\to\infty}\left\{N^{\frac{2}{d-1}}\delta_\Sigma(B_d,\mathcal{P}_{d,k}^{\rm in}(N)) \right\} \geq \frac{1}{2}\operatorname{div}_{d-1}\omega_{d}^{\frac{2}{d-1}}d\overline{W}(B_d).
    \]

    \item[(ii)]  Let $k\in\{\lceil d/2\rceil-1,\ldots,d-1\}$ be fixed. Then
    \[
\liminf_{N\to\infty}\left\{N^{\frac{2}{d-1}}\delta_\Sigma(B_d,\mathcal{P}_{d,k}^{\rm out}(N)) \right\} \geq \frac{1}{2}\operatorname{del}_{d-1}\omega_{d}^{\frac{2}{d-1}}d\overline{W}(B_d).
    \]
    \end{itemize}
\end{corollary}

\begin{remark}
It is known that in the cases of inscribed polytopes with a fixed number of vertices or circumscribed polytopes with a fixed number of facets, these lower bounds give the best possible approximation (up to absolute constants); see \cite[Theorem 5]{BHK}. Corollaries \ref{wills-cor} 
and \ref{large-N-thm-Wills} extend the corresponding lower bounds in \cite[Theorem 5]{BHK} by removing the assumption that the polytopes are simplicial, and in the circumscribed case, our lower bounds hold for an additional 25\% of the $f$-vector.
\end{remark}

\begin{comment}
\begin{remark}
    The dual volume metric and dual total intrinsic volumes were introduced in \cite{BHK}. Using the methods in this paper, the asymptotic lower bounds for the dual volume metric from \cite{BHK} can be extended from $k=0$ to all $k\in\{0,1,\ldots,\lfloor d/2\rfloor\}$ in the case of inscribed polytopes, and from $k=d-1$ to all $k\in\{\lceil d/2\rceil-1,\ldots,d-1\}$ in the case of circumscribed polytopes.
\end{remark}
To illustrate this point we  mention in the next  theorem  the result for the dual volume metric. We recall  that the dual volume metric  is defined as

\begin{theorem}\label{dual metric}

\end{theorem}
\end{comment}
%%%%%%%%%%%%%%%%%%%%%%%%%%
\subsection{Arbitrarily positioned polytopes}

The results in this section will show that for volume and surface area approximation, dropping the restriction that the polytope contains the ball (or vice versa) improves the lower estimates by a factor of dimension.

\subsubsection{Symmetric volume difference}

Our next result extends \cite[Theorem 2]{LSW} from $k=d-1$ to all $k\in\{ \lceil\tfrac{d}{2}\rceil-1,\ldots,d-1\}$.

\begin{theorem}\label{vol-lower-bd}
Let $d\geq 2$ and  $k\in\{ \lceil\tfrac{d}{2}\rceil-1,\ldots,d-1\}$. For every polytope $P_M$ in $\R^d$ with at most $M$ $k$-faces and whose facets all meet the interior of $B_d$, we have
    \[
\vol_d(B_d\triangle P_M)>\frac{1}{4\pi e}\vol_{d-1}(\partial P_M\cap B_d)^{\frac{d+1}{d-1}}\left(1-O\left(\frac{\ln d}{d}\right)\right)M^{-\frac{2}{d-1}}.
    \]
    In particular, if $P_M$ also  satisfies $\vol_{d-1}(\partial P_M\cap B_d)\geq c\,\omega_{d}$ for some absolute constant $c>0$, then
    \[
        \vol_d(B_d\triangle P_M) \geq c_1\kappa_dM^{-\frac{2}{d-1}}
    \]
    for some absolute constant $c_1>0$.
\end{theorem}

To prove Theorem \ref{vol-lower-bd}, we use Lemma \ref{inclusion-chain} and the following lemma, whose proof follows directly from that of  \cite[Theorem 2]{LSW}.

\begin{lemma}\label{LSW-lower-bd}
    For every $d\geq 2$ and for all polytopes  $P_N$  with at most $N$ facets, all of which meet the interior of $B_d$, we have
    \[
\vol_d(B_d\triangle P_N)>\frac{1}{4\pi e}\vol_{d-1}(\partial P_N\cap B_d)^{\frac{d+1}{d-1}}\left(1-O\left(\frac{\ln d}{d}\right)\right)\NN.
    \]
    In particular, if $P_N$ also  satisfies $\vol_{d-1}(\partial P_N\cap B_d)\geq c\cdot\omega_{d}$ for some absolute constant $c>0$, then
    \[
        \vol_d(B_d\triangle P_N) \geq c_1\kappa_d\left(1-O\left(\frac{\ln d}{d}\right)\right)\NN.
    \]
\end{lemma}

 For completeness and the reader's convenience, we include the full details of the proof of Lemma \ref{LSW-lower-bd} in the appendix.

\begin{remark}
K. J. B\"or\"oczky \cite[Theorem A(i)]{Boroczky-2000} proved that if $k\in\{0,\ldots,d-1\}$, $K$ is a convex body in $\R^d$ with $C^2$ boundary, and $P_M$ is a polytope with at most $M$ $k$-faces,  then
\[
\vol_d(K\triangle P_M) >\frac{1}{67\pi e^2}\frac{{\rm as}(K)^{\frac{d+1}{d-1}}}{d}M^{-\frac{2}{d-1}}.
\]
Here, $\operatorname{as}(K)= \int_{\partial K} \kappa_K(x)^{\frac{1}{d+1}}\, d\mathcal{H}^{d-1}(x)$ is the  affine surface area of $K$, where $\kappa_K(x)$ denotes the
(generalized) Gauss curvature at $x\in \partial K$, the boundary of $K$, and
$\mathcal{H}^{d-1}$ is the surface area measure on $\partial K$. 
The definition of the generalized Gauss curvature is due to Alexandroff \cite{Alexandroff} and Busemann--Feller \cite{Buse-Feller}.
In the case $K=B_d$, by Stirling's inequality this yields
\begin{equation}\label{boroczky-lower}
\vol_d(B_d\triangle P_M) \geq \frac{c}{d}\kappa_dM^{-\frac{2}{d-1}}.
\end{equation}
Therefore, in the special case of the Euclidean ball and for $k\in\{\lceil\frac{d}{2}\rceil-1,\ldots,d-1\}$, if $\vol_{d-1}(\partial P_M\cap B_d)\geq c\cdot\omega_{d}$, then Theorem \ref{vol-lower-bd} improves B\"or\"oczky's estimate \eqref{boroczky-lower} by a factor of dimension.  Furthermore, B\"or\"oczky's estimate requires $M$ to be sufficiently large, whereas our lower bound holds for any $M$.
\end{remark}

\subsubsection{Surface area deviation}

For convex bodies $K$ and $L$ in $\R^d$, the \emph{surface area deviation} $\Delta_s(K,L)$ is defined by
\[
\Delta_s(K,L):=\vol_{d-1}(\partial K)+\vol_{d-1}(\partial L)-2\vol_{d-1}(\partial(K\cap L)).
\]
It is not a metric on the set of convex bodies in $\R^d$. When $K\subset L$, it reduces to the \emph{surface area difference}
\[
\Delta_s(K,L) = \vol_{d-1}(\partial L)-\vol_{d-1}(\partial K).
\]

Next, we use \eqref{facets-bd} to extend \cite[Theorem 2]{HSW} from $k=d-1$ to all $\lceil\tfrac{d}{2}\rceil-1\leq k\leq d-1$. The extended result reads as follows.

\begin{theorem}\label{mainThm-general}
    Let $d\geq 2$ and $k\in\{\lceil\tfrac{d}{2}\rceil-1,\ldots, d-1\}$. Let $P_M$ be a polytope with at most $M$ $k$-faces and assume that $o\in\operatorname{int}(P_M)\subset 2B_d$.
 There is an absolute constant $c>0$  such that 
    \[
        \Delta_s(B_d,P_M) \geq c\frac{\vol_{d-1}(\partial P_M)^{\frac{d+1}{d-1}}}{\kappa_{d-1}^{\frac{2}{d-1}}}M^{-\frac{2}{d-1}}.
    \]
\end{theorem}

Hence, we obtain the following.
\begin{corollary}\label{mainThm-2}
    Let $d\geq 2$ and $k\in\{\lceil\tfrac{d}{2}\rceil-1,\ldots, d-1\}$. There is an absolute constant $c>0$ and $M_{d,k}\in\mathbb{N}$ such that for all $M\in\mathbb{N}$ with $M\geq M_{d,k}$ and all polytopes $P_M$ with at most $M$ $k$-dimensional faces,
    \[
        \Delta_s(B_d,P_M) \geq c\,\omega_{d}M^{-\frac{2}{d-1}}.
    \]
\end{corollary}

\begin{proof}
    When $M$ is sufficiently large (say, for all $M\geq M_{d,k}$), for a best-approximating polytope $P_M^*$ we can assume that $o\in\operatorname{int}(P_M^*)$, $\vol_{d-1}(\partial P_M^*)\geq c_1\omega_{d}$ and $P_M^*\subset 2B_d$. Hence, for any polytope $P_M$ in $\R^d$ with at most $M$ $k$-faces, by Theorem \ref{mainThm-general} and Stirling's formula, we get that for all $M\geq M_{d,k}$,
    \begin{align*}
\Delta_s(B_d,P_M) \geq \Delta_s(B_d,P_M^*) \geq c\, c_1^{\frac{d+1}{d-1}}\omega_{d}\left(\frac{\omega_{d}}{\kappa_{d-1}}\right)^{\frac{2}{d-1}}M^{-\frac{2}{d-1}}
\geq c_2\,\omega_{d}M^{-\frac{2}{d-1}}.
    \end{align*}
\end{proof}

Theorem \ref{mainThm-general} will be derived from small modifications to the end of the proof of \cite[Theorem 2]{HSW}. The proof of Theorem \ref{mainThm-general} is also in the appendix.

%%%%%%%%%%%%%%%%%%%%%%%%%%%%%%%%%%%%%%%
\section{Proofs of Theorems \ref{inscribed-mw-cor-kfaces} and \ref{circumscribed-vol-cor-kfaces}}\label{mw-vol-sec}

    In what follows, let $R:\R^d\to\Sp$, $x\mapsto x/\|x\|_2$, denote the radial projection. 

\subsection{Proof of Theorem \ref{inscribed-mw-cor-kfaces}}

Let $d\geq 2$. We aim to prove that for all polytopes $P_N\subset B_d$ with at most $N$ vertices that contain $o$ in their interiors,  we have
    \begin{equation}\label{mw-ineq-to-be-shown}
    \Delta_w(B_d,P_N) \geq \frac{1}{4}\left(\frac{\omega_{d}}{4\kappa_{d-1}}\right)^{\frac{2}{d-1}}\NN.
    \end{equation}

We recall that for a convex body $K \subset\mathbb{R}^d$ such that $o$ is in the interior of $K$, the polar body $K^\circ$ of $K$ is 
\[
K^\circ = \{ y \in \mathbb{R}^d:\,\forall x\in K,  \langle y, x \rangle \leq 1\}. 
\]
    To begin, note that since $P_N\subset B_d$ has at most  $N$ vertices and $o$ is in the interior of $P_N$, the polar polyhedron $P_N^\circ\supset B_d$ has at most $N$ facets, and  $P_N^\circ$ is bounded (and thus is a polytope).  Let 
    \[
    \eta_{d,N}:=\frac{1}{3}\left(\frac{\omega_{d}}{4\kappa_{d-1}N}\right)^{\frac{2}{d-1}},\]
    and define the following subset of the sphere:
    \[
\mathcal{U}_{P_N}:=\left\{u\in\Sp:\, h_{P_N}(u)\leq\frac{1}{1+\eta_{d,N}}\right\}.
    \]

Since $h_{P_N}=\rho_{P_N^\circ}^{-1}$, we have
    \begin{align}\label{UPN-lower-bd}
    \mathcal{U}_{P_N} &= \left\{u\in\Sp:\,\rho_{P_N^\circ}(u)\geq 1+\eta_{d,N}\right\} \nonumber\\
    &=\bigcup_{F\in\mathcal{F}_{d-1}(P_N^\circ)}R\left(F\cap((1+\eta_{d,N})B_d)^c\right)\nonumber\\
    &=\Sp\setminus \bigcup_{F\in\mathcal{F}_{d-1}(P_N^\circ)}R\left(F\cap((1+\eta_{d,N})B_d)\right).
    \end{align}
     By the Cauchy--Schwarz inequality and the fact that $P_N^\circ\supset B_d$, for every facet $F\in\mathcal{F}_{d-1}(P_N^\circ)$ with outer unit normal $\xi_F$, we have
    \begin{align*}
\vol_{d-1}\left(R\left(F\cap((1+\eta_{d,N})B_d)\right)\right) &=\int_{F\cap((1+\eta_{d,N})B_d)}\frac{\langle\xi_F,x\rangle}{\|x\|_2^d}\,dx\\
&\leq \int_{F\cap((1+\eta_{d,N})B_d)}\frac{dx}{\|x\|_2^{d-1}}\\
&\leq \vol_{d-1}\left(F\cap((1+\eta_{d,N})B_d)\right). 
    \end{align*}

Observe that if $F$ is any facet of $P_N^\circ$ with outer unit normal $\xi_F$, then $F\cap((1+\eta_{d,N})B_d)$ is contained in the $(d-1)$-dimensional ball $B_{d-1}(\dist(o,F)\xi_F,\sqrt{2\eta_{d,N}+\eta_{d,N}^2})$ with center $\dist(o,F)\xi_F$ and radius $\sqrt{2\eta_{d,N}+\eta_{d,N}^2}$. Hence,
\begin{align*}
   \vol_{d-1}\left(\bigcup_{F\in\mathcal{F}_{d-1}(P_N^\circ)}R\left(F\cap((1+\eta_{d,N})B_d)\right)\right) &\leq \sum_{F\in\mathcal{F}_{d-1}(P_N^\circ)}\vol_{d-1}\left(R(F\cap((1+\eta_{d,N})B_d))\right)\\
    &\leq \sum_{F\in\mathcal{F}_{d-1}(P_N^\circ)}\vol_{d-1}(F\cap ((1+\eta_{d,N}) B_d))\\
    &\leq \sum_{F\in\mathcal{F}_{d-1}(P_N^\circ)}\vol_{d-1}\left(B_{d-1}(\dist(o,F)\xi_F,\sqrt{2\eta_{d,N}+\eta_{d,N}^2})\right)\\
    &\leq \sum_{F\in\mathcal{F}_{d-1}(P_N^\circ)}\vol_{d-1}\left(B_{d-1}(\dist(o,F)\xi_F,\sqrt{3\eta_{d,N}})\right)\\
    &=N(3\eta_{d,N})^{\frac{d-1}{2}}\kappa_{d-1}
    =\frac{1}{4}\omega_{d}.
\end{align*}
In the last inequality, we used $\eta_{d,N}\leq 1$, which is proved below. Hence, by \eqref{UPN-lower-bd} we get
    \begin{equation}\label{lower-UPN}
\vol_{d-1}(\mathcal{U}_{P_N}) \geq \omega_{d}-\frac{1}{4}\omega_{d}= \frac{3}{4}\omega_{d}.
    \end{equation}

To conclude the proof, we use the previous estimates, the definition of mean width (\ref{width})  and the definition of $\mathcal{U}_{P_N}$ to derive
\begin{align*}
    \Delta_w(B_d,P_N) &=w(B_d)-w(P_N)
    =\frac{2}{\omega_{d}}\int_{\Sp}\left(1-h_{P_N}(u)\right)du\\
    &\geq \frac{2}{\omega_{d}}\int_{\mathcal{U}_{P_N}}\left(1-h_{P_N}(u)\right)du
    \geq\frac{2}{\omega_{d}}\int_{\mathcal{U}_{P_N}}\left(1-\frac{1}{1+\eta_{d,N}}\right)du\\
    &=\frac{2\eta_{d,N}}{1+\eta_{d,N}}\cdot\frac{\vol_{d-1}(\mathcal{U}_{P_N})}{\omega_{d}}
    \geq \frac{3}{2}\cdot\frac{\eta_{d,N}}{1+\eta_{d,N}}
    \geq \frac{3}{4}\eta_{d,N}.
\end{align*}
In the last line, we used the inequality $\eta_{d,N}\leq 1$  again. To see why $\eta_{d,N}\leq 1$ holds, note that for every $d\geq 1$ we have
\[
N\geq d+1> \sqrt{2\pi d}\left(\frac{1}{3}\right)^{\frac{d-1}{2}}\geq \left(\frac{1}{3}\right)^{\frac{d-1}{2}}\frac{\omega_{d}}{\kappa_{d-1}}
>\left(\frac{1}{3}\right)^{\frac{d-1}{2}}\frac{\omega_{d}}{4\kappa_{d-1}},
\]
where by Stirling's inequality, $\sqrt{2\pi d}\geq \omega_{d}/\kappa_{d-1}$. Rearranging terms in the preceding inequality and raising both sides to the power $2/(d-1)$, we obtain $\eta_{d,N}\leq 1$.

Finally, to extend the result to all $k$-faces with $k\in\{0,1,\ldots,\lfloor d/2\rfloor\}$, we combine the inequalities $\Delta_w(B_d,P_N)\geq \frac{3}{4}\eta_{d,N}$ and \eqref{vertices-bd}. \qed 
%%%%%%%%%%%%%%%%%
\subsection{Proof of Theorem \ref{circumscribed-vol-cor-kfaces}}

 Glasauer and Gruber \cite{glasgrub} proved a formula relating the mean width difference of a convex body and a polytope with a weighted symmetric difference metric of the polar body and the polar polytope. In the special case of the Euclidean unit ball and an inscribed polytope which contains the origin in its interior, their formula states that
    \[
w(B_d)-w(P)=\frac{2}{\omega_{d}}\int_{P^\circ\setminus B_d}\|x\|_2^{-(d+1)}\,dx.
    \]
    Since $\|x\|_2\geq 1$ for all $x\in P^\circ\setminus B_d$, we have
    \begin{equation}\label{polarity-formula-1}
w(B_d)-w(P)\leq \frac{2}{\omega_{d}}\cdot\vol_d(P^\circ\setminus B_d).
    \end{equation}

    Let $P_M\supset B_d$ be a polytope with at most $M$ $k$-faces. Since $P_M$ is a polytope, it is bounded, so $P_M^\circ\subset B_d$ contains the origin in its interior. Moreover,  $f_k(P_M)=f_{d-1-k}(P_M^\circ)$, so $P_M^\circ$ has at most $M$ $(d-1-k)$-dimensional faces. Since $k\in\{\lceil d/2\rceil-1,\ldots,d-1\}$, we have $d-1-k\in\{0,1,\ldots,\lfloor d/2\rfloor\}$. Thus, by \eqref{polarity-formula-1} and Theorem \ref{inscribed-mw-cor-kfaces}, we obtain
    \begin{align*}
        \vol_d(P_M\setminus B_d)\geq\frac{\omega_{d}}{2}\left(w(B_d)-w(P_M^\circ)\right) 
        \geq \frac{\omega_{d}}{8}\left(\frac{\omega_{d}}{4\kappa_{d-1}}\right)^{\frac{2}{d-1}}M^{-\frac{2}{d-1}}.
    \end{align*}
\qed 

%%%%%%%%%%%%%%%%%%%%%%%%%%%%%%

\section{Asymptotic inequalities}\label{asymptotic-section}

Let $K$ be a convex body in $\R^d$ with $C_+^2$ curvature. Glasauer and Gruber \cite{glasgrub} proved that
\begin{equation}\label{mw-asymptotic-glas-grub}
    \begin{split}
\lim_{N\to\infty} N^{\frac{2}{d-1}}\inf\{\Delta_w(K,P_N)&:\,P_N \subset K\text{ has at most }N\text{ vertices}\}\\
&=\frac{\operatorname{div}_{d-1}}{\omega_{d}}\left(\int_{\partial K}\kappa(x)^{\frac{d}{d+1}}\,d\mathcal{H}(x)\right)^{\frac{d+1}{d-1}}.
\end{split}
\end{equation}
Choosing $K=B_d$, this yields
\begin{equation}\label{mw-asymptotic-glas-grub-ball}
    \begin{split}
\lim_{N\to\infty} N^{\frac{2}{d-1}}\inf\{\Delta_w(B_d,P_N):\,P_N \subset B_d\text{ has at most }N\text{ vertices}\}
= \operatorname{div}_{d-1}\cdot\omega_{d}^{\frac{2}{d-1}}.
\end{split}
\end{equation}

\begin{theorem}\label{mw-asymptotic}
    \begin{align*}
\lim_{N\to\infty} N^{\frac{2}{d-1}}\inf\{\Delta_w(B_d,P_N):\,P_N \subset B_d\text{ has at most }N\text{ vertices}\} 
\geq \frac{d-1}{d+1}\left(\frac{2\omega_{d}}{(d+1)\kappa_{d-1}}\right)^{\frac{2}{d-1}}.
    \end{align*}
\end{theorem}

\begin{proof}
    Let $c\in(0,1)$ be fixed. Let $P_N^*$ be a best-approximating polytope which attains the infimum. We now define
    \[
\eta_{d,N,c}:=\frac{1}{2}\left(\frac{c\,\omega_{d}}{\kappa_{d-1}N}\right)^{\frac{2}{d-1}}.
    \]
Following the proof of Theorem \ref{inscribed-mw-cor-kfaces}, we obtain
\begin{align*}
    &\vol_{d-1}\left(\bigcup_{F\in\mathcal{F}_{d-1}((P_N^*)^\circ)}R(F\cap((1+\eta_{d,N,c})B_d))\right) \\&\leq \sum_{F\in\mathcal{F}_{d-1}((P_N^*)^\circ)}\vol_{d-1}\left(B_{d-1}(\dist(o,F)\xi_F,\sqrt{2\eta_{d,N,c}+\eta^2_{d,N,c}})\right)\\
    &=N(2\eta_{d,N,c})^{\frac{d-1}{2}}\left(1+\frac{\eta_{d,N,c}}{2}\right)^{\frac{d-1}{2}}\kappa_{d-1}
    =\left(1+\frac{\eta_{d,N,c}}{2}\right)^{\frac{d-1}{2}}\cdot c\,\omega_{d}.
\end{align*}
Hence, setting $\mathcal{U}_{P_N^*}(c):=\{u\in\Sp:\, h_{P_N^*}(u)\leq\tfrac{1}{1+\eta_{d,N,c}}\}$, we get 
\[
\vol_{d-1}(\mathcal{U}_{P_N^*}(c)) \geq \left(1-c\left(1+\frac{\eta_{d,N,c}}{2}\right)^{\frac{d-1}{2}}\right)\omega_{d}.
\]
Now fix $\varepsilon>0$. Since $\eta_{d,N,c}\to 0$ as $N\to\infty$, there exists $N(d,c,\varepsilon)$ such that 
\[
\forall N\geq N(d,c,\varepsilon),\quad \left(1+\frac{\eta_{d,N,c}}{2}\right)^{\frac{d-1}{2}} \leq \exp\left(\frac{d-1}{4}\eta_{d,N,c}\right)\leq 1+\varepsilon.
\]
Therefore,
\[
\forall N\geq N(d,c,\varepsilon),\quad \vol_{d-1}(\mathcal{U}_{P_N^*}(c)) \geq  \left(1-c(1+\varepsilon)\right)\omega_{d}.
\]
Thus, for all $N\geq N(d,c,\varepsilon)$ we have
\begin{align*}
\Delta_w(B_d,P_N^*) &\geq \frac{2\eta_{d,N,c}}{1+\eta_{d,N,c}}\cdot\frac{\vol_{d-1}(\mathcal{U}_{P_N^*}(c))}{\omega_{d}}    \geq \frac{2\eta_{d,N,c}}{1+\eta_{d,N,c}}\left(1-c(1+\varepsilon)\right).
\end{align*}
Since $\eta_{d,N,c}\to 0$ as $N\to\infty$, for all sufficiently large $N$,
\begin{align*}
    \frac{2\eta_{d,N,c}}{1+\eta_{d,N,c}}=2\eta_{d,N,c}\left(1-\eta_{d,N,c}+\eta_{d,N,c}^2-\eta_{d,N,c}^3-\cdots\right)=2\eta_{d,N,c}+O(\eta_{d,N,c}^2).
\end{align*}
In fact, 
\[
\forall\eta_{d,N,c}\in[0,1/2], \quad \frac{2\eta_{d,N,c}}{1+\eta_{d,N,c}} \geq 2\eta_{d,N,c}-2\eta_{d,N,c}^2.
\]
Thus, for all sufficiently large $N$ we have
\[
\Delta_w(B_d,P_N^*) \geq 2\eta_{d,N,c}\left(1-c(1+\varepsilon)\right)-C(d,c)N^{-\frac{4}{d-1}}
\]
where $C(d,c)>0$ is a constant that can be explicitly computed. Using the definition of $\eta_{d,N,c}$ and rearranging terms, we get that for all sufficiently large $N$,
\[
\Delta_w(B_d,P_N^*) \geq c^{\frac{2}{d-1}}\left(1-c(1+\varepsilon)\right)\left(\frac{\omega_{d}}{N\kappa_{d-1}}\right)^{\frac{2}{d-1}}-C(d,c)N^{-\frac{4}{d-1}}.
\]

Now multiply both sides by $N^{\frac{2}{d-1}}$ and let $N\to\infty$ (for fixed $c,\varepsilon$). Since the error term $C(d,c)\NN$ tends to 0 as $N$ tends to infinity, we obtain
\[
\liminf_{N\to\infty}N^{\frac{2}{d-1}}\Delta_w(B_d,P_N^*) \geq  c^{\frac{2}{d-1}}\left(1-c(1+\varepsilon)\right)\left(\frac{\omega_{d}}{\kappa_{d-1}}\right)^{\frac{2}{d-1}}.
\]
Letting $\varepsilon\to 0$, we get that for each fixed $c\in(0,1)$, 
\[
\liminf_{N\to\infty}N^{\frac{2}{d-1}}\Delta_w(B_d,P_N^*) \geq  c^{\frac{2}{d-1}}\left(1-c\right)\left(\frac{\omega_{d}}{\kappa_{d-1}}\right)^{\frac{2}{d-1}}.
\]

Finally, we optimize in $c\in[0,1]$. The maximum of the function $f:[0,1]\to[0,1]$ defined by $f(c)=c^{\frac{2}{d-1}}(1-c)$ is attained at $c^*=\frac{2}{d+1}\in(0,1)$, with 
\[
f(c^*) = \frac{d-1}{d+1}\left(\frac{2}{d+1}\right)^{\frac{2}{d-1}}.
\] 
Thus,
\[
\liminf_{N\to\infty}N^{\frac{2}{d-1}}\Delta_w(B_d,P_N^*) \geq  \frac{d-1}{d+1}\left(\frac{2\omega_{d}}{(d+1)\kappa_{d-1}}\right)^{\frac{2}{d-1}}.
\]
By the definition of the best-approximating polytope $P_N^*$, this is the claimed inequality.
\end{proof}

Following the proof of Theorem \ref{circumscribed-vol-cor-kfaces}, we obtain the following
\begin{theorem}\label{vol-asymptotic}
    \begin{align*}
\lim_{N\to\infty} N^{\frac{2}{d-1}}\inf\{\vol_d(Q_N\setminus B_d):\,Q_N \supset B_d\text{ has at most }N\text{ facets}\} 
\geq \frac{\omega_{d}}{2}\,\frac{d-1}{d+1}\left(\frac{2\omega_{d}}{(d+1)\kappa_{d-1}}\right)^{\frac{2}{d-1}}.
    \end{align*}
\end{theorem}

%%%%%%%%%%%%%%%%%%%%%%%%%
\subsection{On the Dirichlet--Voronoi tiling number $\operatorname{div}_{d-1}$}

Hoehner and Kur \cite[Theorem 1.4]{HK-DCG} proved that \begin{equation}\label{HK-div}
\left|\operatorname{div}_{d-1}-(2\pi e)^{-1}(d+\ln d)\right|=O(1).
\end{equation}
Following the arguments in their proof, one finds that the constant term $O(1)$ is at most
\begin{equation}\label{c1-eqn}
    c_1 = \frac{2d}{d+1}\ln\left(\frac{8d}{d+1}(2\pi d)^{\frac{1}{d-1}}\right)\sim 2\ln 8.
\end{equation}
The proof in \cite{HK-DCG} used a polytope $P_b$ with at most $N$ facets (all with the same height $t_{d,N}\in(0,1)$), which is neither contained in nor contains the ball, and inflated it so that it circumscribes the ball. Next, they used known results on the approximation of the ball in these two settings to derive their estimate for $\operatorname{div}_{d-1}$. We refer the reader to \cite{HK-DCG} for the details. 

Using the results in this paper, we give a new,  simpler proof of this estimate which uses only one mode of approximation (circumscribed), rather than two as in \cite{HK-DCG}.

\begin{theorem}\label{div-estimate}
    \[
    \frac{d-1}{d+1} \left(\frac{2}{(d+1)\kappa_{d-1}}\right)^{\frac{2}{d-1}} \leq  \operatorname{div}_{d-1} 
    \leq \frac{2}{d-1}\cdot\frac{\Gamma\left(\frac{2}{d-1}\right)}{\kappa_{d-1}^{\frac{2}{d-1}}}.
    \]
\end{theorem}

This gives the same asymptotic estimate for $\operatorname{div}_{d-1}$ as the one given in  \cite{HK-DCG}, that is $\operatorname{div}_{d-1}=\frac{d}{2\pi e}(1+O(\tfrac{\ln d}{d}))$, where the error term $O(\tfrac{\ln d}{d})$ can be computed explicitly. It can be shown that our estimate is better than the one in \cite{HK-DCG} for $d$ less than about $10^{15}$.

\begin{proof}
Combining Theorem \ref{mw-asymptotic} with \eqref{mw-asymptotic-glas-grub}, we obtain
\begin{equation}
    \operatorname{div}_{d-1}\cdot\omega_{d}^{\frac{2}{d-1}} \geq \frac{d-1}{d+1}\left(\frac{2\omega_{d}}{(d+1)\kappa_{d-1}}\right)^{\frac{2}{d-1}}.
\end{equation}
Dividing both sides of this inequality by $\omega_{d}^{\frac{2}{d-1}}$, we obtain the lower bound.

  We now prove the upper bound.  A special case of a result of M\"uller \cite{Muller-mw} states that if $X_1,\ldots,X_N$ are random points selected independently and uniformly from $\partial B_d$, and $P_N$ is the convex hull of $\{X_1,\ldots,X_N\}$,  then
\begin{equation}\label{muller-mw}
    \begin{split}
        \lim_{N\to\infty}N^{\frac{2}{d-1}}\left(w(B_d)-\E(w(P_N))\right)
        =\frac{2}{d-1}\left(\frac{\omega_{d}}{\kappa_{d-1}}\right)^{\frac{2}{d-1}}\Gamma\left(\frac{2}{d-1}\right).
    \end{split}
\end{equation}
Combining this with \eqref{mw-asymptotic-glas-grub-ball}, we get
\begin{align*}
      \operatorname{div}_{d-1}\cdot\omega_{d}^{\frac{2}{d-1}}&=\lim_{N\to\infty} N^{\frac{2}{d-1}}\inf\{\Delta_w(B_d,P_N):\,P_N \subset B_d\text{ has at most }N\text{ vertices}\}\\
&\leq \lim_{N\to\infty}N^{\frac{2}{d-1}}\left(w(B_d)-\E(w(P_N))\right)
=\frac{2}{d-1}\left(\frac{\omega_{d}}{\kappa_{d-1}}\right)^{\frac{2}{d-1}}\Gamma\left(\frac{2}{d-1}\right).
\end{align*}
Dividing both sides of this inequality by $\omega_{d}^{\frac{2}{d-1}}$, we obtain the upper bound.
\end{proof}

%%%%%%%%%%%%%%%%%%%%%%%%%%%%%%%%
\section{Proof of Theorem \ref{mainThm}}\label{mainThm-sec}

The strategy of the proof is as follows. To prove (i), we apply  the extended isoperimetric inequality (\ref{extisoineq}), the nonasymptotic lower bound in Theorem \ref{inscribed-mw-cor-kfaces} for the mean width difference of an inscribed polytope with $N$ vertices and the Euclidean ball, and  \eqref{vertices-bd}  to obtain the result for $k\in\{0,1,\ldots,\lfloor d/2\rfloor\}$ and all $j\in\{1,\ldots,d\}$.  Similarly, to prove (ii), we apply the extended isoperimetric inequality, the nonasymptotic lower bound in Theorem \ref{circumscribed-vol-cor-kfaces} for the volume difference of a circumscribed polytope with $N$ facets and the Euclidean ball, and \eqref{facets-bd} to obtain the result for $k\in\{\lceil d/2\rceil-1,\ldots,d-1\}$ and all $j\in\{1,\ldots,d\}$. 

(i) Let $P_M\subset B_d$ with $o\in\operatorname{int}(P_M)$. Using the extended isoperimetric inequality, it was shown in \cite[Inequality (55)]{BHK} that for every $j\in\{1,\ldots,d\}$,
\begin{equation}\label{BHK-lower-2}
    \frac{V_j(B_d)-V_j(P_M)}{V_j(B_d)} \geq 1-\left(1-\frac{V_1(B_d)-V_1(P_M)}{V_1(B_d)}\right)^j=1-\left(1-\frac{w(B_d)-w(P_M)}{w(B_d)}\right)^j
\end{equation}
where in the equality we used the fact that for any convex body $K$ in $\R^d$, the first intrinsic volume is a constant multiple of the mean width, namely, 
$V_1(K)=\frac{\omega_{d}}{2\kappa_{d-1}}w(K)$. Using \eqref{BHK-lower-2} and Theorem \ref{inscribed-mw-cor-kfaces}, we get
\begin{align*}
    \delta_j(B_d,P_M) &\geq V_j(B_d)\left(1-\left(1-\frac{w(B_d)-w(P_M)}{w(B_d)}\right)^j\right)\\
        &\geq \left(1-\left(1-\frac{1}{4}\left(\frac{\omega_{d}}{4\kappa_{d-1}}\right)^{\frac{2}{d-1}}M^{-\frac{2}{d-1}}\right)^j\right)V_j(B_d).
\end{align*}
Now using the inequality $(1-(1-x)^j) \geq jx-\frac{j(j-1)}{2}x^2$, which holds for all $x\in[0,1)$ and all integers $j\geq 1$, we get
\begin{align*}
    \delta_j(B_d,P_M) 
        \geq \frac{j}{4}\left(\frac{\omega_{d}}{4\kappa_{d-1}}\right)^{\frac{2}{d-1}}V_j(B_d)M^{-\frac{2}{d-1}}
        \left(1-\frac{j-1}{8}\left(\frac{\omega_{d}}{4\kappa_{d-1}}\right)^{\frac{2}{d-1}}M^{-\frac{2}{d-1}}\right).
\end{align*}

Now to prove the second part of (i), note that for any $c\in(0,1)$,
\[
1-\frac{j-1}{8}\left(\frac{\omega_{d}}{4\kappa_{d-1}}\right)^{\frac{2}{d-1}}M^{-\frac{2}{d-1}} \geq c
\]
if and only if $M\geq \frac{\omega_{d}}{4\kappa_{d-1}}\left(\frac{j-1}{8(1-c)}\right)^{\frac{d-1}{2}}$.

(ii) Using the extended isoperimetric inequality, it was shown in \cite[Inequality (58)]{BHK} that for any polytope $P_M\supset B_d$ and any $j\in\{1,\ldots,d\}$, we have
    \begin{equation}
        \left(1+\frac{V_d(P_M)-V_d(B_d)}{V_d(B_d)}\right)^{\frac{j}{d}} -1 \leq \frac{V_j(P_M)-V_j(B_d)}{V_j(B_d)}.
    \end{equation}
Using this estimate and Theorem \ref{circumscribed-vol-cor-kfaces}, we  obtain
\begin{align*}
    \delta_j(B_d,P_M) &\geq V_j(B_d)\left(\left(1+\frac{V_d(P_M)-V_d(B_d)}{V_d(B_d)}\right)^{\frac{j}{d}} -1\right)\\
    &\geq\left(\left(1+ \frac{d}{8}\left(\frac{\omega_{d}}{4\kappa_{d-1}}\right)^{\frac{2}{d-1}}M^{-\frac{2}{d-1}}\right)^{\frac{j}{d}}-1\right)V_j(B_d).
\end{align*}

Using the inequality $(1+x)^r\geq 1+rx-\frac{r(1-r)}{2}x^2$, which holds for all $x\geq 0$ and all $r\in(0,1)$, with 
$x=\alpha_{d,M}:=\frac{d}{8}\left(\frac{\omega_{d}}{4\kappa_{d-1}}\right)^{\frac{2}{d-1}}M^{-\frac{2}{d-1}}\geq 0$ and $r=j/d\in(0,1]$ we get
\begin{align*}
\delta_j(B_d,P_M) \geq \frac{j}{8}\left(\frac{\omega_{d}}{4\kappa_{d-1}}\right)^{\frac{2}{d-1}}M^{-\frac{2}{d-1}}
\left(1-\frac{d-j}{16}\left(\frac{\omega_{d}}{4\kappa_{d-1}}\right)^{\frac{2}{d-1}}M^{-\frac{2}{d-1}}\right)V_j(B_d).
\end{align*}

The second assertion of (ii) is proved in the same way as the second assertion of part (i). 
 \qed

%%%%%%%%%%%%%%%%%%%%%%%%%%%%%%%%%

\section{Hausdorff approximation}\label{hausdorff-sec}

In this section, we obtain lower bounds for the Hausdorff approximation of the Euclidean ball by polytopes with a fixed number of $k$-faces. Recall that for convex bodies $K$ and $L$ in $\R^d$, the \emph{Hausdorff distance} $d_H(K,L)$ is defined by
\[
d_H(K,L) = \inf\{\lambda\geq 0:\,K\subset L+\lambda B_d, L\subset K+\lambda B_d\}.
\]
For background on the Hausdorff approximation of convex bodies by polytopes, we refer the reader to, for example, \cite{Bronshtein-survey,glasauer-schneider,GruberI, PSoSchW-2025} and the references therein.

\begin{theorem}[Hausdorff approximation by inscribed polytopes]\label{inscribed-Hausdorff-kfaces}
    Let $d\geq 2$ and $k\in\{0,1,\ldots,d-1\}$. For all polytopes $P_M\subset B_d$ with at most $M$ $k$-faces, we have
    \begin{equation}
        d_H(B_d,P_M) \geq \frac{1}{2}\left(\frac{\vol_{d-1}(\partial P_M)}{4\omega_{d}}\right)^{\frac{2}{d-1}}M^{-\frac{2}{d-1}}.
    \end{equation}
   \end{theorem}

By Stirling's inequality, this implies that for all polytopes $P_M\subset B_d$ satisfying $\vol_{d-1}(\partial P_M) \geq c\vol_{d-1}(\partial B_d)$ for some absolute constant $c>0$, we have $d_H(B_d,P_M)\geq c_1 M^{-\frac{2}{d-1}}$ for all $k\in\{0,1,\ldots,d-1\}$. When $k\in\{0,d-1\}$, Theorem \ref{inscribed-Hausdorff-kfaces} gives the best possible estimate, up to an absolute constant (see \cite{Bronshtein-survey} and the references therein).  To the best of our knowledge, the cases $k\in\{1,\ldots,d-2\}$ are new for general dimensions $d\geq 2$.

\begin{theorem}[Hausdorff approximation by circumscribed polytopes]\label{circumscribed-Hausdorff-kfaces}
    Let $d\geq 2$ and $k\in\{0,1,\ldots,d-1\}$.  For all polytopes satisfying $B_d\subset P_M\subset\sqrt{2}B_d$ with at most $M$ $k$-faces, we have
    \begin{equation}
        d_H(B_d,P_M) \geq \frac{1}{3}\left(\frac{\vol_{d-1}(\partial P_M)}{4\omega_{d}}\right)^{\frac{2}{d-1}}M^{-\frac{2}{d-1}}.
    \end{equation}
\end{theorem}

By Stirling's inequality, this implies that for all polytopes satisfying $B_d\subset P_M\subset \sqrt{2}B_d$ and all $k\in\{0,1,\ldots,d-1\}$, we have $d_H(B_d,P_M)\geq cM^{-\frac{2}{d-1}}$ for some absolute constant $c>0$. When $k\in\{0,d-1\}$, Theorem \ref{circumscribed-Hausdorff-kfaces} gives the best possible estimate, up to an absolute constant (see \cite{Bronshtein-survey} and the references therein). To the best of our knowledge, the cases $k\in\{1,\ldots,d-2\}$ are new for general dimensions $d\geq 2$.
%%%%%%%%%%%%%%%%%%%%%%%
\subsection{Proof of Theorem \ref{inscribed-Hausdorff-kfaces}}

\subsubsection{Case 1: $k\in\{\lceil\tfrac{d}{2}\rceil-1,\ldots,d-1\}$}

This case follows from Lemma \ref{inclusion-chain} and the following lemma.

\begin{lemma}\label{inscribed-Hausdorff-facets}
    Let $d\geq 2$. For all polytopes $P_N\subset B_d$ with at most $N$ facets, we have
     \begin{equation}
        d_H(B_d,P_N) \geq \frac{1}{2}\left(\frac{\vol_{d-1}(\partial P_N)}{4\kappa_{d-1}}\right)^{\frac{2}{d-1}}N^{-\frac{2}{d-1}}.
    \end{equation}
\end{lemma}

\begin{proof}
Let $\eta_{P_N}:=\frac{1}{2}\left(\frac{\vol_{d-1}(\partial P_N)}{4\kappa_{d-1}N}\right)^{\frac{2}{d-1}}$. Suppose by way of contradiction that for every $F\in\mathcal{F}_{d-1}(P_N)$, we have $\dist(o,F)>1-\eta_{P_N}$. Then each facet $F$ of $P_N$ is contained in a $(d-1)$-dimensional ball of radius $\sqrt{2\eta_{P_N}-\eta_{P_N}^2}<\sqrt{2\eta_{P_N}}$. This implies
\begin{align*}
\vol_{d-1}(\partial P_N) = \sum_{F\in\mathcal{F}_{d-1}(P_N)}\vol_{d-1}(F)
&<\sum_{F\in\mathcal{F}_{d-1}(P_N)}\vol_{d-1}(\sqrt{2\eta_{P_N}}B_{d-1})\\
&=N(2\eta_{P_N})^{\frac{d-1}{2}}\kappa_{d-1}
=\frac{\vol_{d-1}(\partial P_N)}{4},
\end{align*}
a contradiction. Thus, there exists $F_0\in\mathcal{F}_{d-1}(P_N)$ such that $\dist(o,F_0) \leq 1-\eta_{P_N}$. Therefore,
\begin{align*}
    d_H(B_d,P_N) \geq \max\left\{1-\dist(o,F):\,F\in\mathcal{F}_{d-1}(P_N)\right\}
    \geq 1-\dist(o,F_0)
    \geq 1-(1-\eta_{P_N})
    =\eta_{P_N}.
\end{align*}
\end{proof}

\begin{remark}
For $c\in(0,1)$, define $\eta_{P_N,c}:=\frac{1}{2}\left(c\cdot\frac{\vol_{d-1}(\partial P_N)}{\kappa_{d-1}N}\right)^{\frac{2}{d-1}}$. Then as before, we obtain 
\[
\vol_{d-1}(\partial P_N) < c\cdot\vol_{d-1}(\partial P_N),
\]
a contradiction. Thus, as before, $d_H(B_d,P_N)\geq \eta_{P_N,c}$. Since $c\in(0,1)$ was arbitrary, by the continuity of the function $x\mapsto x^{\frac{2}{d-1}}$ on $[0,1]$, we get
\[
d_H(B_d,P_N)\geq \lim_{c\to 1^-} \eta_{P_N,c}=\eta_{P_N,1}=\frac{1}{2}\left(\frac{\vol_{d-1}(\partial P_N)}{\kappa_{d-1}}\right)^{\frac{2}{d-1}}\NN.
\]
The same remark applies to Lemma \ref{circumscribed-vertices-Hausdorff} below (and we can essentially improve the constant factor $1/3$ to $1/2$ there as well).
\end{remark}

\subsubsection{Case 2: $k\in\{0,1,\ldots,\lfloor\tfrac{d}{2}\rfloor\}$}

This case follows from Lemma \ref{inclusion-chain} and the following lemma.

\begin{lemma}\label{inscribed-Hausdorff-vertices}
    Let $d\geq 2$. For all polytopes $P_N\subset B_d$ with at most $N$ vertices, we have
     \begin{equation}
        d_H(B_d,P_N) \geq \frac{1}{2}\left(\frac{\vol_{d-1}(\partial P_N)}{4\omega_{d}}\right)^{\frac{2}{d-1}}N^{-\frac{2}{d-1}}.
    \end{equation}
\end{lemma}

\begin{proof}
    The proof is similar to that of Lemma \ref{inscribed-Hausdorff-facets}. In the vertices case, we define $r_{P_N}:=\frac{1}{2}\left(\frac{\vol_{d-1}(\partial P_N)}{4\omega_{d}N}\right)^{\frac{2}{d-1}}$. Note that if $\dist(o,F)>1-r_{P_N}$ for all facets $F$ of $P_N$, then
    \[
    P_N \subsetneq \bigcup_{v\in\mathcal{F}_0(P_N)}B_d(v,\sqrt{2r_{P_N}})
    \]
    where $B_d(v,\sqrt{2r_{P_N}})$ is the $d$-dimensional Euclidean ball with center $v$ and radius $\sqrt{2r_{P_N}}$. 
    Since $P_N$ is convex and is a strict subset of the union of the balls, this implies
    \[
\vol_{d-1}(\partial P_N) < \sum_{v\in\mathcal{F}_0(P_N)}\vol_{d-1}\left(\partial B_d(v,\sqrt{2r_{P_N}})\right)=N(2r_{P_N})^{\frac{d-1}{2}}\omega_{d}
=\frac{\vol_{d-1}(\partial P_N)}{4},
    \]
    a contradiction. Thus, there exists a facet $F_0$ of $P_N$ such that $\dist(o,F_0)\leq 1-r_{P_N}$. Therefore,
\begin{align*}
    d_H(B_d,P_N) \geq \max\left\{1-\dist(o,F):\,F\in\mathcal{F}_{d-1}(P_N)\right\}
    \geq 1-\dist(o,F_0)
    \geq 1-(1-r_{P_N})
    =r_{P_N}.
\end{align*}
\end{proof}

%%%%%%%%%%%%%%
\subsection{Proof of Theorem \ref{circumscribed-Hausdorff-kfaces}}

\subsubsection{Case 1: $k\in\{0,1,\ldots,\lfloor\tfrac{d}{2}\rfloor\}$} We follow a dual argument to the proof of Theorem \ref{inscribed-Hausdorff-kfaces}. This case follows from \eqref{vertices-bd} and the following lemma.

\begin{lemma}\label{circumscribed-vertices-Hausdorff}
    Let $d\geq 2$. For all polytopes $B_d\subset P_N\subset\sqrt{2}B_d$ with at most $N$ vertices, we have
     \begin{equation}
        d_H(B_d,P_N) \geq \frac{1}{3}\left(\frac{\vol_{d-1}(\partial P_N)}{4\omega_{d}}\right)^{\frac{2}{d-1}}N^{-\frac{2}{d-1}}.
    \end{equation}
\end{lemma}

    Let $r_{P_N}:=\frac{1}{3}\left(\frac{\vol_{d-1}(\partial P_N)}{4\omega_{d}N}\right)^{\frac{2}{d-1}}$. To prove the last lemma, we will need the following
\begin{lemma}\label{rpN-lemma}
        If $B_d\subset P_N\subset \sqrt{2}B_d$ has at most $N$ vertices, then $r_{P_N}\leq 1$.
    \end{lemma}

    \begin{proof}
        Note that for every $d\geq 1$, since $P_N\supset B_d$, we have
\begin{align*}
N\geq d+1>\left(\frac{2}{3}\right)^{\frac{d-1}{2}}\cdot\frac{1}{4}= \left(\frac{2}{3}\right)^{\frac{d-1}{2}}\frac{\omega_{d}}{4\omega_{d}}
= \left(\frac{1}{3}\right)^{\frac{d-1}{2}}\frac{\vol_{d-1}(\partial (\sqrt{2}B_d))}{4\omega_{d}}
\geq\left(\frac{1}{3}\right)^{\frac{d-1}{2}}\frac{\vol_{d-1}(\partial P_N)}{4\omega_{d}}.
\end{align*}
Rearranging terms and raising both sides of the inequality to the power $2/(d-1)$, we obtain $r_{P_N}\leq 1$.
    \end{proof}

\begin{proof}[Proof of Lemma \ref{circumscribed-vertices-Hausdorff}] 
Suppose by way of contradiction that for every $v\in\mathcal{F}_0(P_N)$, we have $\|v\|_2<1+r_{P_N}$. Then the union $\mathscr{N}(v)$ of the facets incident to $v$ is contained  in a $d$-dimensional ball of radius $\sqrt{2r_{P_N}+r_{P_N}^2}<\sqrt{3r_{P_N}}$ centered at $v$, where in the inequality we used Lemma \ref{rpN-lemma}. This implies
\begin{align*}
\vol_{d-1}(\partial P_N) &=\vol_{d-1}\left(\bigcup_{v\in\mathcal{F}_0(P_N)}\bigcup_{F\in\mathscr{N}(v)}F\right)\leq\sum_{v\in\mathcal{F}_0(P_N)}\vol_{d-1}\left(\bigcup_{F\in\mathscr{N}(v)}F\right)\\
&\leq \sum_{v\in\mathcal{F}_0(P_N)}\vol_{d-1}\left(\partial (\sqrt{3r_{P_N}}B_d)\right)
= N(3r_{P_N})^{\frac{d-1}{2}}\omega_{d}=\frac{\vol_{d-1}(\partial P_N)}{4},
\end{align*}
a contradiction. Thus, there exists $v_0\in\mathcal{F}_0(P_N)$ such that $\|v_0\|_2 \geq 1+r_{P_N}$. Therefore,
\begin{align*}
    d_H(B_d,P_N) \geq \max\left\{\|v\|_2-1:\,v\in\mathcal{F}_0(P_N)\right\}
    \geq \|v_0\|_2-1
    \geq (1+r_{P_N})-1
    =r_{P_N}.
\end{align*}
\end{proof}

\subsubsection{Case 2: $k\in\{\lceil\tfrac{d}{2}\rceil-1,\ldots,d-1\}$}  This case follows from \eqref{facets-bd} and the following lemma.

\begin{lemma}
    Let $d\geq 2$. For all polytopes $B_d\subset P_N\subset\sqrt{2}B_d$ with at most $N$ facets, we have
     \begin{equation}
        d_H(B_d,P_N) \geq \frac{1}{3}\left(\frac{\vol_{d-1}(\partial P_N)}{4\kappa_{d-1}}\right)^{\frac{2}{d-1}}N^{-\frac{2}{d-1}}.
    \end{equation}
\end{lemma}

\begin{proof}
%    The proof is similar to that of Lemma \ref{circumscribed-vertices-Hausdorff}, only now we use the factor $\eta_{P_N}$ instead of $r_{P_N}$. The proof that $\eta_{P_N}\leq 1$ when $P_N\subset\sqrt{2}B_d$ is similar to the proof of Lemma \ref{rpN-lemma}. 
    Let $\eta_{P_N}:=\frac{1}{3}\left(\frac{\vol_{d-1}(\partial P_N)}{4\kappa_{d-1}N}\right)^{\frac{2}{d-1}}$. 
    As in the proof of Lemma \ref{rpN-lemma}, we have $\eta_{P_N}\leq 1$. Suppose by way of contradiction that $d_H(B_d,P_N)<\eta_{P_N}$. Then $P_N\subset (1+\eta_{P_N})B_d$. Since $B_d\subset P_N$, every facet of $P_N$ lies in a hyperplane whose distance from $o$ is at least $1$. Hence, each facet is contained in a $(d-1)$-dimensional ball of radius at most
    \[
        \sqrt{(1+\eta_{P_N})^2-1}
        =\sqrt{2\eta_{P_N}+\eta_{P_N}^2}
        \leq \sqrt{3\eta_{P_N}}.
    \]
    Therefore,
    \[
    \vol_{d-1}(\partial P_N)
    \leq
    N(3\eta_{P_N})^{\frac{d-1}{2}}\kappa_{d-1}
    =\frac{\vol_{d-1}(\partial P_N)}{4},
    \]
    a contradiction. Thus, $d_H(B_d,P_N)\geq \eta_{P_N}$.
\end{proof}

\begin{remark}
    A result of K. J. B\"or\"oczky \cite[Theorem A(ii)]{Boroczky-2000} applied to the Euclidean unit ball states that if $k\in\{0,1,\ldots,d-1\}$ and $P_M^*$ is a best-approximating polytope for $B_d$ with respect to the Hausdorff distance, then for all sufficiently large $M$,
    \[
d_H(B_d,P_M^*) > \frac{d}{34e\pi}\omega_{d}^{\frac{2}{d-1}}M^{-\frac{2}{d-1}}\geq \frac{1}{17}M^{-\frac{2}{d-1}}
    \]
where in the last inequality we used $\omega_{d}^{\frac{2}{d-1}}\geq\frac{2\pi e}{d}$, which holds by Stirling's inequality.   In B\"or\"oczky's result, there is no restriction on the position of the approximating polytope, whereas our lower bounds are for inscribed and circumscribed polytopes. Moreover, in contrast to B\"or\"oczky's result, our lower bounds do not require $M$ to be sufficiently large.
\end{remark}

%%%%%%%%%%%%%%%%%%%%%%%%%%%%%%

\section{Discussion}\label{sec-discussion}

In the case of inscribed and circumscribed polytopes, in each part of Theorem \ref{mainThm} we  established lower bounds for half of the $f$-vector. Here we describe an approach that may yield lower bounds of the same order for the remaining values of $k$. 

    Let $P_N\supset B_d$ be a polytope with at most $N$ vertices, and set $t_{P_N}:=c_1\left(\frac{\vol_{d-1}(\partial P_N)}{4\omega_{d}N}\right)^{2/(d-1)}$. We conjecture that
    \begin{equation}\label{conjectured-ineq}
        \vol_{d-1}(\partial P_N\cap ((1+t_{P_N})B_d)^c) \geq c_2\,\omega_{d}
    \end{equation}
    for some absolute constants $c_1,c_2>0$. If this inequality is true, then arguments similar to those in this paper will yield the same lower bounds for the remaining $k$ in each case (inscribed and circumscribed). Inequality \eqref{conjectured-ineq} is plausible, as it says that the vertices of a circumscribed polytope cannot all be ``too close'' to the ball, and it is certainly true for $d=2$, as the proof of Theorem \ref{inscribed-mw-cor-kfaces} above shows.

%%%%%%%%%%%%%%%%%%%%%%%%%%%%%%%%%%%
\section*{Acknowledgments}

This material is based upon work supported by the National Science Foundation under Grant No. DMS-1929284 while the authors were in residence at the Institute for Computational and Experimental Research in Mathematics in Providence, RI, during the Harmonic Analysis and Convexity program. 
\par
In addition, Elisabeth Werner was supported by the National Science Foundation under grants DMS-2103482 and DMS-2506790.
\vskip 2mm
We would  like to thank the anonymous referee for their careful reading of the manuscript and for their many
suggestions which helped to improve our text.

%%%%%%%%%%%%%%%%%%%%%%%%%%%%%%%%%%%%%%
\section{Appendix: Proofs of Lemma \ref{LSW-lower-bd} and Theorem \ref{mainThm-2}}\label{sec-appendix}

\subsection{Proof of Lemma \ref{LSW-lower-bd}}

To prove Lemma \ref{LSW-lower-bd}, we make very minor adjustments to the proof of Theorem 2 in \cite{LSW}. The only obstruction is that we cannot assume that our approximating polytope with $N$ facets is best-approximating for the Euclidean ball, which was assumed in the proof of  \cite[Theorem 2]{LSW}. This is because our theorem considers best-approximating polytopes with a fixed number of $k$-faces, and a best-approximating polytope with at most $N$ $k$-faces may be different, in general, from a best-approximating polytope with at most $N$ facets. In the proof given in \cite{LSW}, they proved and used the fact that for any best-approximating polytope $P_N^*$ with at most $N$ facets, we have $\vol_{d-1}(F\cap B_d)=\vol_{d-1}(F\cap B_d^c)$ for all $F\in\mathcal{F}_{d-1}(P_N^*)$; as a consequence, $\vol_{d-1}(\partial P_N^*\cap B_d)=\vol_{d-1}(\partial P_N^*\cap B_d^c)=\frac{1}{2}\vol_{d-1}(\partial P_N^*)$. 

Let $P_N$ be any polytope with at most $N$ facets whose facets all meet the interior of  $B_d$. Observe that, with $h_F=1-\dist(o,F)$, we have
\begin{equation}\label{LSW-1}
    \vol_d(P_N\triangle B_d) \geq \frac{1}{d}\sum_{F\in\mathcal{F}_{d-1}(P_N)}h_F\vol_{d-1}(F\cap B_d).
\end{equation}
Since each facet of $P_N$ meets the interior of $B_d$, we have $h_F> 0$. Let $r_F:=\left(\frac{\vol_{d-1}(F\cap B_d)}{\kappa_{d-1}}\right)^{\frac{1}{d-1}}$ denote the volume radius of $F\cap B_d$. Let $\widetilde{h}_F$ denote the height of the cap whose base is parallel to $F$ and has volume  $\vol_{d-1}(F\cap B_d)=r_F^{d-1}\kappa_{d-1}$. 
Note that  $\widetilde{h}_F\leq h_F$. Furthermore, $(1-\widetilde{h}_F)^2+r_F^2=1$, which implies 
\[
\widetilde{h}_F=1-\sqrt{1-r_F^2} \geq \frac{1}{2}r_F^2 =\frac{1}{2}\left(\frac{\vol_{d-1}(F\cap B_d)}{\kappa_{d-1}}\right)^{\frac{2}{d-1}},
\]
where in the first inequality we used $\sqrt{1-x^2}\leq 1-\frac{1}{2}x^2$ for $x\in[0,1]$. Thus, by \eqref{LSW-1} we get
\begin{align*}
     \vol_d(P_N\triangle B_d) \geq  \frac{1}{2d\kappa_{d-1}^{\frac{2}{d-1}}}\sum_{F\in\mathcal{F}_{d-1}(P_N)}(\vol_{d-1}(F\cap B_d))^{\frac{d+1}{d-1}}.
\end{align*}
%By Stirling's inequality, we have $\kappa_{d-1}^{\frac{2}{d-1}}=\frac{d}{2\pi e}\left(1+O(\tfrac{\ln d}{d})\right)$. 
By Stirling's formula, we have
\[
    \kappa_{d-1}^{\frac{2}{d-1}}
    =\frac{2\pi e}{d-1}\left(1+O\left(\frac{\ln d}{d}\right)\right).
\]
Hence,
\[
    \frac{1}{2d\kappa_{d-1}^{\frac{2}{d-1}}}
    =\frac{1}{4\pi e}\left(1-O\left(\frac{\ln d}{d}\right)\right).
\]
Using this and the previous estimate, we get
\begin{equation}\label{LSW-2}
     \vol_d(P_N\triangle B_d) \geq \frac{1}{4\pi e}\left(1-O\left(\frac{\ln d}{d}\right)\right)\sum_{F\in\mathcal{F}_{d-1}(P_N)}(\vol_{d-1}(F\cap B_d))^{\frac{d+1}{d-1}}.
\end{equation}
Next, by H\"older's inequality with conjugate exponents $p,q\geq 1$ satisfying $\frac{1}{p}+\frac{1}{q}=1$, 
\[
\sum_{F\in\mathcal{F}_{d-1}(P_N)}\vol_{d-1}(F\cap B_d) \leq \left(\sum_{F\in\mathcal{F}_{d-1}(P_N)}(\vol_{d-1}(F\cap B_d))^p\right)^{\frac{1}{p}}N^{\frac{1}{q}}.
\]
Choosing $p=\frac{d+1}{d-1}$ and $q=\frac{d+1}{2}$, we obtain
\[
\sum_{F\in\mathcal{F}_{d-1}(P_N)}\vol_{d-1}(F\cap B_d) \leq \left(\sum_{F\in\mathcal{F}_{d-1}(P_N)}(\vol_{d-1}(F\cap B_d))^{\frac{d+1}{d-1}}\right)^{\frac{d-1}{d+1}}N^{\frac{2}{d+1}}.
\]
Therefore,
\begin{align*}
\sum_{F\in\mathcal{F}_{d-1}(P_N)}(\vol_{d-1}(F\cap B_d))^{\frac{d+1}{d-1}} &\geq \left(\sum_{F\in\mathcal{F}_{d-1}(P_N)}\vol_{d-1}(F\cap B_d)\right)^{\frac{d+1}{d-1}}\NN\\
&=(\vol_{d-1}(\partial P_N\cap B_d))^{\frac{d+1}{d-1}}\NN.\\
\end{align*}
Combining the previous estimate with \eqref{LSW-2}, we get
\[
\vol_d(P_N\triangle B_d)>\frac{1}{4\pi e}\left(1-O\left(\frac{\ln d}{d}\right)\right)\vol_{d-1}(\partial P_N\cap B_d)^{\frac{d+1}{d-1}}\NN,
\]
as desired.

For the second statement, note that by  the assumption that $\vol_{d-1}(\partial P_N\cap B_d)\geq c\cdot\omega_{d}$, we have
\[
(\vol_{d-1}(\partial P_N\cap B_d))^{\frac{d+1}{d-1}}\NN\geq c^{\frac{d+1}{d-1}}\omega_{d}^{\frac{d+1}{d-1}}\NN.
\]Therefore, combining the last estimate with \eqref{LSW-2} and using the identity $\omega_{d}=d\kappa_d$, we get
\begin{align*}
    \vol_d(B_d\triangle P_N) &> \frac{c^{\frac{d+1}{d-1}}}{4\pi e}\omega_{d}^{\frac{d+1}{d-1}}\left(1-O\left(\frac{\ln d}{d}\right)\right)\NN\\
    &=\frac{c^{\frac{d+1}{d-1}}}{4\pi e}d^{\frac{2}{d-1}}\cdot d\kappa_d^{\frac{2}{d-1}}\cdot \kappa_d\left(1-O\left(\frac{\ln d}{d}\right)\right)\NN.
\end{align*}
Finally, note that $\frac{c^{\frac{d+1}{d-1}}}{4\pi e} d^{\frac{2}{d-1}} \geq c_1>0$ for some absolute constant $c_1$, and by Stirling's inequality, 
$d\kappa_d^{\frac{2}{d-1}}\geq c_2>0$ for some absolute constant $c_2$. Therefore, 
\[
\vol_d(P_N\triangle B_d)\geq c_3\kappa_d\left(1-O\left(\frac{\ln d}{d}\right)\right)\NN
\]
for some absolute constant $c_3>0$. \qed

\subsection{Proof of Theorem \ref{mainThm-2}}

We will need  \cite[Proposition 10]{HSW} and \cite[Proposition 12]{HSW}, which we state here for convenience.

\begin{lemma}[\cite{HSW}]\label{HSW-props}
\begin{itemize}
    \item[(i)] For all $d\in\mathbb{N}$ with $d\geq 2$, all $N\in\mathbb{N}$, and all polytopes $P_N$ in $\R^d$ with at most $N$ facets, 
    \[
        \vol_{d-1}(\partial B_d\cap P_N^c) - \vol_{d-1}(\partial P_N\cap B_d) \geq \frac{c_1(\vol_{d-1}(B_d\cap \partial P_N))^{\frac{d+1}{d-1}}}{\kappa_{d-1}^{\frac{2}{d-1}}}\NN.
    \]

    \item[(ii)] For all $d\in\mathbb{N}$ with $d\geq 2$, all $N\in\mathbb{N}$, and all polytopes $P_N$ in $\R^d$ with at most $N$ facets and $o\in{\rm int}(P_N)\subset 2B_d$,
    \[
    \vol_{d-1}(\partial P_N\cap B_d^c)-\vol_{d-1}(\partial B_d\cap P_N) \geq \frac{c_2(\vol_{d-1}(\partial P_N\cap B_d^c))^{\frac{d+1}{d-1}}}{\kappa_{d-1}^{\frac{2}{d-1}}}\NN.
    \]
\end{itemize}\end{lemma}

The first step is to extend this result from $k=d-1$ to all $k\in\{\lceil\tfrac{d}{2}\rceil-1,\ldots, d-1\}$ as follows.

\begin{corollary}\label{cor1}
\begin{itemize}
    \item[(i)] For all $d\in\mathbb{N}$ with $d\geq 2$ and all $k\in\{\lceil\tfrac{d}{2}\rceil-1,\ldots, d-1\}$, all $M\in\mathbb{N}$, and all polytopes $P_M$ in $\R^d$ with at most $M$ $k$-faces, 
    \[
        \vol_{d-1}(\partial B_d\cap P_M^c) - \vol_{d-1}(\partial P_M\cap B_d) \geq \frac{c_1(\vol_{d-1}(B_d\cap \partial P_M))^{\frac{d+1}{d-1}}}{\kappa_{d-1}^{\frac{2}{d-1}}}M^{-\frac{2}{d-1}}.
    \]

    \item[(ii)] For all $d\in\mathbb{N}$ with $d\geq 2$ and all $k\in\{\lceil\tfrac{d}{2}\rceil-1,\ldots, d-1\}$, all $M\in\mathbb{N}$, and all polytopes $P_M$ in $\R^d$ with at most $M$ $k$-faces and $o\in{\rm int}(P_M)\subset 2B_d$,
    \[
    \vol_{d-1}(\partial P_M\cap B_d^c)-\vol_{d-1}(\partial B_d\cap P_M) \geq \frac{c_2(\vol_{d-1}(\partial P_M\cap B_d^c))^{\frac{d+1}{d-1}}}{\kappa_{d-1}^{\frac{2}{d-1}}}M^{-\frac{2}{d-1}}.
    \]
\end{itemize}\end{corollary}

\begin{proof}
    Since $f_k(P_M)\leq M$, we have $N=f_{d-1}(P_M)\leq f_k(P_M)\leq M$ by \eqref{facets-bd}. Thus the result follows immediately from  Lemma \ref{HSW-props} and another application of \eqref{facets-bd}.
\end{proof}

Assume that $o\in{\rm int}(P_M)\subset 2B_d$. Theorem \ref{mainThm-general} now follows from Corollary \ref{cor1} and H\"older's inequality:
\begin{align*}
    \Delta_s(B_d,P_M) &\geq  \vol_{d-1}(\partial B_d\cap P_M^c) - \vol_{d-1}(\partial P_M\cap B_d)
    + \vol_{d-1}(\partial P_M\cap B_d^c)-\vol_{d-1}(\partial B_d\cap P_M)\\
    &\geq \frac{c_1(\vol_{d-1}(B_d\cap \partial P_M))^{\frac{d+1}{d-1}}}{\kappa_{d-1}^{\frac{2}{d-1}}}M^{-\frac{2}{d-1}}
    +\frac{c_2(\vol_{d-1}(\partial P_M\cap B_d^c))^{\frac{d+1}{d-1}}}{\kappa_{d-1}^{\frac{2}{d-1}}}M^{-\frac{2}{d-1}}\\
    &\geq \frac{c_3(\vol_{d-1}(\partial P_M))^{\frac{d+1}{d-1}}}{\kappa_{d-1}^{\frac{2}{d-1}}}M^{-\frac{2}{d-1}}.
\end{align*}\qed

%%%%%%%%%%%%%%%%%%%%%%%%%%%%%%%%%
\bibliographystyle{plain}
\bibliography{main}

%%%%%%%%%%%%%%

\vspace{3mm}
\noindent
Steven Hoehner\\
\noindent {\sc Department of Mathematics \& Computer Science, Longwood University, 201 High Street, Farmville, VA 23909, U.S.A.}\\
\noindent {\it E-mail address:} {\tt hoehnersd@longwood.edu}
\vskip 3mm
\noindent
Carsten Sch\"utt\\
\noindent {\sc Department of Mathematics, University of Kiel, Heinrich-Hecht-Platz 6, 24118 Kiel, Germany} \\
\noindent {\it E-mail address:} {\tt schuett@math.uni-kiel.de} 
\vskip 3mm
\noindent
Elisabeth Werner\\
\noindent {\sc Department of Mathematics, Case Western Reserve University, 2145 Adelbert Road, Cleveland, OH 44106, U.S.A.}\\
\noindent {\it E-mail address:} {\tt elisabeth.werner@case.edu}

%%%%%%%%%%%%%%%%%%%%%%%%%%%%%%%%
\end{document}